\documentclass[11pt]{amsart}
\usepackage{amsmath, amsxtra, amssymb, color}
\usepackage{verbatim}
\usepackage{graphicx}
\usepackage[cmtip,all]{xy}
\usepackage{xypic}
\usepackage{setspace}

\definecolor{linkred}{rgb}{0.7,0.2,0.2}

\usepackage[colorlinks, 
linkcolor=blue, 
citecolor= linkred,
pagebackref=true]{hyperref}

\usepackage{multirow, tikz}

\usepackage[margin=1.18in]{geometry}
                 
\theoremstyle{plain}
\newtheorem{main-theorem}{Main Theorem}

\numberwithin{equation}{section}
\theoremstyle{plain}
\newtheorem{theorem}{Theorem}[section]
\newtheorem{prop}[theorem]{Proposition}
\newtheorem{corollary}[theorem]{Corollary}

\newtheorem{lemma}[theorem]{Lemma}

\theoremstyle{definition}

\newtheorem{definition}[theorem]{Definition}
\newtheorem{construction}[theorem]{Construction}
\newtheorem{example}[theorem]{Example}
\newtheorem{remark}[theorem]{Remark}

\def\ra{\rightarrow}

\newcommand{\gitq}{/\hspace{-0.25pc}/}

\def\co{\colon\thinspace} %macro to use in f\co \rightarrow Y 

\DeclareMathOperator{\PEff}{\overline{Eff}}
\DeclareMathOperator{\Nef}{Nef}
\DeclareMathOperator{\Pic}{Pic}

\DeclareMathOperator{\spec}{Spec}

\DeclareMathOperator{\Sym}{Sym}

\DeclareMathOperator{\ord}{ord}

\DeclareMathOperator{\Exc}{Exc}

\def\Hn1{\mathcal{H}_{n,1}}

\renewcommand\min{\text{min}}

\def\irr{\text{irr}}
\def\red{\text{red}}

\def\DD{\mathbb{D}}
\def\sl{\mathfrak{sl}}

\def\A{\mathcal{A}}

\def\C{\mathcal{C}}
\def\D{\mathcal{D}}

\def\O{\mathcal{O}}

\def\M{\overline{M}}
\newcommand\Mg[1]{\overline{\mathcal{M}}_{#1}}
\newcommand\tM[1]{\widetilde{M}_{#1}}
\newcommand\modp[2]{\langle #1\rangle_{#2}}
\def\Br{\mathrm{Br}}

\def\N{\mathcal{N}}
\def\L{\mathcal{L}}

\def\X{\mathcal{X}}
\def\MM{\mathcal{M}}
\def\Y{\mathcal{Y}}
\def\Z{\mathcal{Z}}

\def\sigman{\{\sigma_i\}_{i=1}^{n}}

\def\QQ{\mathbb{Q}}
\def\ZZ{\mathbb{Z}}

\def\PP{\mathbb{P}}
\def\ZZ{\mathbb{Z}}

\def\FF{\mathbb{F}}

\def\CC{\mathbb{C}}

\def\DD{\mathbb{D}}
\def\EE{\mathbb{E}}
\def\HH{\mathrm{H}}

\DeclareMathOperator\SL{SL}
\DeclareMathOperator\spn{span}

\def\nb{\nobreakdash}

\begin{document}
\title{Cyclic covering morphisms on $\M_{0,n}$}% and nef divisors on $\M_{0,n}$}
\author{Maksym Fedorchuk}
\email{mfedorch@math.columbia.edu}
\address{Department of Mathematics, Columbia University, 2990 Broadway, New York, NY 10027}
%\classification{14H10, 14E30} 
%\keywords{moduli of curves, nef cone, cyclic cover} 

\begin{abstract} 
We study cyclic covering morphisms from $\M_{0,n}$ to
moduli spaces of unpointed stable curves of positive genus or compactified 
moduli spaces of principally polarized abelian varieties. 
Our main application is a construction of new semipositive 
vector bundles and nef divisors on $\M_{0,n}$, with a view toward the F-conjecture. 
In particular, we construct new extremal rays of 
$\Nef(\M_{0,n}/S_{n})$. We also find an alternate description of all
$\sl$ level $1$ conformal blocks divisors on $\M_{0,n}$.
\end{abstract}
\maketitle

%\setcounter{tocdepth}{1}
%\tableofcontents

\section{Introduction}
\label{S:introduction}
The purpose of this paper is to describe a new geometric 
construction of nef line bundles on $\M_{0,n}$, the moduli space of stable $n$-pointed rational curves. 
The construction is motivated by the F-conjecture, posed by Fulton, 
that describes $\Nef(\M_{0,n})$ as the polyhedral
cone, called the {\em F-cone}, cut out by the hyperplanes dual to $1$-dimensional boundary strata in 
$\M_{0,n}$. The F-conjecture was first addressed in \cite{keel-mckernan}
and has attracted a great deal of interest since;
see \cite{GKM, FG, Gib, Larsen, fontanari}. In particular, by \cite{GKM} the $S_n$-symmetric variant
of the F-conjecture implies the analog of the F-conjecture for $\M_{g}$ formulated by Faber 
in \cite{faber}.

Our main tool is the {\em cyclic covering morphism} that associates to a stable $n$\nobreakdash-pointed
rational curve the cyclic degree $p$ cover branched at the marked points ($p$ needs to divide $n$).
The literature on the subject of cyclic covers of a projective line is vast.  
We mention recent papers of Chen \cite{dawei-quadratic}, 
Bouw and M\"oller \cite{bouw-moller}, Eskin, Kontsevich, and Zorich \cite{eskin-kontsevich-zorich}, McMullen \cite{mcmullen}, Moonen, and Oort \cite{moonen, moonen-oort}.
In particular, \cite{eskin-kontsevich-zorich} studies the eigenbundle decomposition of 
the Hodge bundle over a family of cyclic covers of $\PP^1$ with $4$ branch points. 
So far, families of cyclic covers were used to study the geometry of  
the moduli spaces of stable curves, abelian differentials, of principally polarized abelian varieties, etc.
For example, cyclic covers are used to study                
special subvarieties in the moduli space of abelian varieties that arise from families of cyclic covers in 
\cite{moonen, moonen-oort}. 
The novelty of our approach is that we use the cyclic covering construction in a way that bears on 
the cone of nef divisors of $\M_{0,n}$. That this simple construction elucidates the structure 
of the symmetric nef cone of $\M_{0,n}$ comes as a pleasant surprise. %, but mystifies.
 
We now present three theorems that give a flavor of our results. Our most general results are contained in 
the main body of this paper. We also refer the reader to Section \ref{S:n=6}
where the illustrative example of $\M_{0,6}$ is worked out. 

Briefly, for every $p\mid n$, we define the cyclic $p$-covering morphism
$f_{n,p}\co \M_{0,n} \ra \M_{g}$, where $g=(n-2)(p-1)/2$, that associates to a stable 
$n$-pointed rational curve its cyclic degree $p$ cover branched over $n$ marked points. 
After passing to a root stack of $\M_{0,n}$, the cyclic $p$-covering morphism lifts to $\Mg{g}$. 
Our first observation concerning cyclic covering morphisms gives a determinantal
description of $\sl_{n}$ level $1$ conformal blocks bundles on $\M_{0,n}$ recently studied
by Fakhruddin in \cite{fakh}, and Arap, Gibney, Stankewicz, and Swinarski in \cite{agss} (see also \cite{ags}); we follow notation of
\cite{agss}. 

\smallskip
\noindent 
\textbf{Theorem A.} (char $0$) {\em Let $\EE$ be the restriction of the Hodge bundle over 
$\Mg{g}$ 
to $f_{n,p}(\M_{0,n})$. There is an eigenbundle decomposition 
$
\EE=\bigoplus_{j=1}^{p-1} \EE_j
$
with respect to a natural $\mu_p$-action on $\EE$. The vector bundle 
$f_{n,p}^*(\EE_j)$ is semipositive on $\M_{0,n}$ and 
$\det (f_{n,p}^*(\EE_j))=\DD^1_{n,jn/p}/p$, where $\DD^1_{n,jn/p}$ 
is the symmetric $\sl_{n}$ level $1$ conformal blocks divisor associated 
to the fundamental weight $w_{jn/p}$. 
}

\smallskip

By a recent work of Arap, Gibney, Stankewicz, and Swinarski \cite{agss}, 
the $\sl_{n}$ level $1$ conformal blocks divisor $\DD^1_{n,j}$
generates an extremal ray of the nef cone of $\M_{0,n}/S_n$ for every fundamental weight $w_j$, $j=1,\dots, n-1$. 
Theorem A, and its generalization Theorem \ref{T:theorem-AA}, %presented in the main body of this paper, 
provides an algebro-geometric
construction of these line bundles and gives an independent 
proof of nefness of $\DD^1_{n,j}$. %(but not of their extremality in the nef cone).
By a recent work of Giansiracusa \cite{gian}, the line bundle $\DD^1_{n,j}$ is a pullback to $\M_{0,n}$ 
of a polarization on a GIT quotient of the compactified
parameter space of $n$ points lying on rational normal curves of degree $j-1$ in $\PP^{j-1}$. We 
observe that Theorem A expresses the Hodge class $c_1(f^*_{n,p}(\EE))$ as an effective
linear combination of the conformal blocks divisors 
$\DD^1_{n,jn/p}$, $j=1,\dots, p-1$. The following result 
gives a geometric interpretation of this Hodge class.
% (we retain the notation of Theorem A). 
Its relation with the GIT quotients of Giansiracusa is a complete mystery.\footnote{Except in the case of $p=2$ where
$\DD^1_{n,n/2}\sim f^*_{n,2}(\lambda)$, and $p=3$
where $\DD^1_{n,n/3}=\DD^1_{n,2n/3}\sim f^*_{n,3}(\lambda)$; see Proposition \ref{P:new-formula}.}

\smallskip
\noindent 
\textbf{Theorem B.} % (char $0$)
 {\em %(Let $p$ be a prime.)
The line bundle $\lambda_{n, p}:=\det(f^*_{n,p}(\EE))$ is semiample on $\M_{0,n}$
and defines a morphism to the Satake compactification of $\A_{(n-2)(p-1)/2}$. It contracts the 
boundary divisors $\Delta_{k}$ with $p\mid k$, and its divisor class is 
%$$\lambda_{n, p}=\frac{p^2-1}{12p}\left(\psi-\Delta+\sum_{p\mid k} \Delta_k\right).$$}
\[
f_{n,p}^*(\lambda) =\frac{p}{12}\left(\left(1-\frac{1}{p^2}\right)\psi-\sum_{k} \left(1-\frac{\gcd(k,p)^2}{p^2}\right)\Delta_k\right).
\]
}
\smallskip

Theorem A shows that all extremal rays of $\Nef(\M_{0,n}/S_n)$ generated by 
$\sl_n$ level $1$ conformal blocks bundles $\DD^1_{n,j}$ arise from 
the cyclic covering construction. Theorem \ref{T:theorem-AA} in the sequel shows that {\em every} $\sl$ level $1$ 
conformal blocks divisor on $\M_{0,n}$ arises in a similar fashion via a generalized  
cyclic covering morphism.

%As we have indicated, every extremal ray of $\Nef(\M_{0,n}/S_n)$
The following theorem produces a new extremal ray of $\Nef(\M_{0,n}/S_n)$. It
lies outside of the cone generated by the determinants of conformal blocks bundles
computed so far \cite{swinarski-personal}.

\smallskip
\noindent 
\textbf{Theorem C.} %(char $0$)
 {\em Suppose $3\mid n$. The divisor $f^*_{n,3}(9\lambda-\delta_{\irr})$ is an 
extremal ray of $\Nef(\M_{0,n}/S_n)$.
}

\smallskip Theorem C is proved in Section \ref{S:new}.
We believe that there are analogs of Theorem C that produce new extremal rays of $\Nef(\M_{0,n}/S_n)$
for other values of $p$: We prove one for $p=5$ in Proposition \ref{T:p=5}, producing nef divisors on $\Nef(\M_{0,n}/S_n)$
that were previously unattainable with other techniques. (Almost definitely, these divisors generate extremal
rays of $\Nef(\M_{0,n}/S_n)$, for $5\mid n$.)

Taken together, Theorems A and C 
raise a fascinating possibility that the extremal rays of the symmetric nef cone of $\M_{0,n}$ can
be understood by studying families of cyclic covers. This is the subject of our ongoing investigation. 
 
\subsection{Cyclic covering morphisms on $\M_{0,6}$}
\label{S:n=6}
We proceed to describe an illustrative example of our approach in the case of $\tM{0,6}:=\M_{0,6}/S_6$. The Neron-Severi space of $\tM{0,6}$ is generated by divisor classes $\Delta_{2}$ and $\Delta_{3}$. There are only two F-curves: $F_{2,2,1,1}$ %=[-1, 2]$ 
and $F_{3,1,1,1}$. %=[3,-1]$, 
So the extremal rays of the F-cone are spanned by %their duals 
$$D_1=2\Delta_2+\Delta_3 \ \text{  and   } D_2=\Delta_2+3\Delta_3.$$
The divisor $D_{2}$ is easily seen to be semiample: it defines the morphism to $\Sym^6(\PP^1)\gitq \SL(2)$. %(the GIT quotient is with respect to the symmetric linearization). 
This morphism contracts every F-curve of type $(3,1,1,1)$, 
and more generally contracts $\Delta_{3}$ to the unique point in the GIT quotient corresponding to strictly semistable 
orbits. 

To prove that $D_{1}$ is semiample, we introduce the morphism $f_{2}\co \tM{0,6} \ra \M_{2}$ that 
sends a stable $6$-pointed rational curve to its cyclic cover of degree $2$ totally ramified over $6$ marked points. 
For example, a $\PP^{1}$ marked by points $[x_{i}:1]$ is mapped to the genus $2$ curve 
$$
y^{2}=\prod_{i=1}^{6}(x-x_{i}).
$$
This morphism extends to the boundary because $\M_{2}$ is proper and the indeterminacy of $f_2$ 
at the boundary of $\tM{0,6}$ is a priori at worst finite. 
If we take now a stable $6$-pointed rational curve in the $1$-dimensional boundary stratum of type $(2,2,1,1)$, then the 
corresponding double cover has three rational components and its stabilization is a $2$-nodal rational curve. In 
particular, the normalization of the double cover does not vary in moduli.  %accounted by the cross-ratio. 
Consider now the Hodge
class $\lambda\in \Pic(\M_{2})$. It is well-known to define the morphism from $\M_2$ 
to the Satake compactification of $\A_{2}$, also known as the quotient of the {\em Igusa quartic} 
%$$ \subset \PP^4$$
by the action of $S_6$.
Pulling back $\lambda$ via $f_{2}$ we obtain a semiample divisor on $\tM{0,6}$ that intersects 
every F-curve of type $(2,2,1,1)$ in degree $0$. It follows that 
$$D_{1}=2\Delta_2+\Delta_3\sim f_{2}^{*}\lambda,$$
where $\sim$ stands for numerical proportionality.

This is the essence of our method. Two aspects of the described construction can now be varied. 
One is the degree of the cyclic covering, the other is the line bundle that we pullback from a moduli space of positive 
genus curves. To illustrate the first point, consider the morphism $f_{3}\co \tM{0,6}\ra \M_{4}$ that now sends a 
stable $6$-pointed rational curve to its cyclic cover of degree $3$ totally ramified over $6$ marked points:
a $\PP^{1}$ marked by points $[x_{i}:1]$ is mapped to the genus $4$ curve 
$$
y^{3}=\prod_{i=1}^{6}(x-x_{i}).
$$
Every curve in the family of $\mu_{3}$-covers over an F-curve of type $(3,1,1,1)$ 
has an elliptic component with $j$-invariant $0$ 
attached at three points to another elliptic component with $j$-invariant $0$. % (living over the backbone). 
Thus, the degree of $\lambda\in \Pic(\M_{4})$ on such a family is $0$. It follows that 
$$D_{2}\sim f_{3}^{*}\lambda.$$
In particular, the GIT quotient 
of $6$ unordered points on $\PP^1$ has another geometric interpretation: It is (the normalization 
of) the image of the cyclic $3$-covering morphism from $\tM{0,6}$ to the Satake compactification 
of principally polarized abelian varieties of dimension $4$. It is also the
quotient of the {\em Segre cubic} by the action of $S_6$. %a dual of the Igusa quartic considered above.
By \cite[p.257]{prokhorov}, this quotient is isomorphic to $\PP(2,4,5,6)$.

To illustrate the second point, we return to the morphism $f_{2}\co \tM{0,6} \ra \M_{2}$. It is well-known that the line
bundle $12\lambda-\delta_{\irr}$ is nef on $\M_{2}$ and has degree $0$ on any family whose moving 
components are all of genus $1$. % (see \cite{GKM}). 
Observe now that the moving component of the family of 
$\mu_{2}$-covers over an F-curve of type $(3,1,1,1)$ is of genus $1$. It follows that
$$
D_{2}\sim f_{2}^{*}(12\lambda-\delta_{\irr}).
$$
Finally, we know from Theorem C that the line bundle $9\lambda-\delta_{\irr}$ is nef 
on the locus of $\mu_{3}$-covers in $\M_{4}$ and that 
$$
D_{1}\sim f_{3}^{*}(9\lambda-\delta_{\irr}).
$$

\subsection{Notation and conventions}\label{S:notation}

We work over an algebraically closed field of characteristic $0$, which we denote by $\CC$. 
Some results, such as computations of Section \ref{S:appendix}, extend to sufficiently 
high positive characteristic. 

\smallskip
\noindent
We denote by $\tM{0,n}$ the quotient scheme $\M_{0,n}/S_n$. We have
$\Pic(\tM{0,n})=\ZZ\{\Delta_2,\Delta_3,\dots, \Delta_{\lfloor n/2}\}.$

\smallskip
\noindent
We denote by $\modp{a}{p}$ the representative in $\{0,1,\dots, p-1\}$ of the residue of $a$ modulo $p$. 

\smallskip
\noindent If $\vec{d}=(d_1,\dots, d_n)$ and $I\subset \{1,\dots, n\}$, we set $d(I):=\sum_{i\in I} d_i$.
 
 \smallskip
\noindent
An F-curve of type $(a,b,c,d)$ in $\M_{0,n}$
is a family of stable $n$-pointed rational curves 
obtained from the universal family $(\C; \sigma_1,\sigma_2,\sigma_3,\sigma_4)$ 
over $\M_{0,4}$ by identifying sections $\sigma_1,\sigma_2,\sigma_3,\sigma_4$ with sections of constant stable
$(a+1)$, $(b+1)$, $(c+1)$, and $(d+1)$-pointed rational curves.\footnote{We do not require constant families to be 
maximally degenerate.} For $t\in \PP^1$, we call the $4$-pointed 
component $\C_t$ a {\em backbone}; we also denote $p_A:=\sigma_1(t)$, $p_B:=\sigma_2(t)$, $p_C:=\sigma_3(t)$, 
$p_D:=\sigma_4(t)$. 
%The curves attached to the backbone $\C_t$ at points $p_A,p_B,p_C,p_D$ are called {\em teeth}.  
The {\em class} of any F-curve in $\tM{0,n}$ corresponding to the decomposition $n=a+b+c+d$ is denoted by $F_{a,b,c,d}$. 
%With some abuse of language, 
We also call $F_{a,b,c,d}$ an F-curve of type $(a,b,c,d)$. In the non-symmetric case, an F-curve corresponding 
to the partition $\{1,\dots, n\}=I_1\cup I_2\cup I_3 \cup I_4$ is denoted $F_{I_1,I_2,I_3,I_4}$. 
A divisor that intersects every F-curve non-negatively is called {\em F-nef}. The F-nef divisors form
an {\em F-cone} inside $N^1(\M_{0,n})$. Fulton's conjecture posits that the F-cone of $\M_{0,n}$ is equal 
to the nef cone.

\smallskip
\noindent
Given $\QQ$-Cartier divisors $D_1$ and $D_2$ on a projective variety $X$, we say that $D_1\sim D_2$ if $D_1$ and 
$D_2$ generate the same ray in $N^1(X)$, or, equivalently, if $D_1\equiv_{num} cD_2$ for some $c>0$.
\begin{comment}
\begin{remark}\label{R:defines} 
We say that a line bundle 
$\L$ on $X$ defines a morphism to $Y$ if there is a morphism $f\co X\ra Y$ and 
an ample line bundle $\MM$ on $Y$ such that $\L=f^*\MM$. This is admittedly an informal terminology. In order to be 
precise, we note that if $\L$ defines $f\co X\ra Y$, then  
the birational model of $X$ associated to $\L$ is the normalization of $f(X)$.
\end{remark}
\end{comment}
\noindent
Finally, we refer to \cite{keel-mckernan} and \cite{HM} for
 the intersection theory on $\tM{0,n}$ and $\M_{0,n}$. 

\begin{comment}
\subsection{Symmetric vs. non-symmetric divisors}
Nothing in our treatment calls for working with $S_n$-symmetric
divisors on $\M_{0,n}$. However, 
it is well-known that combinatorial complexity of the F-cone increases rapidly with $n$. This is explained by the size of the Neron-Severi space of $\M_{0,n}$. 
On the other hand, the Picard number of $\tM{0,n}$ is $\lfloor n/2\rfloor$. 
For these reasons we restrict ourselves mostly to  
$S_{n}$-symmetric divisors on $\M_{0,n}$, or equivalently, to the study of divisors on the scheme quotient $\tM{0,n}:=\M_{0,n}/S_{n}$. This case is still very interesting
because the affirmative solution of the F-conjecture for $\tM{0,g}$ would give a complete description of $\Nef(\M_g)$ by \cite{GKM}.  
\end{comment}

\section{Cycling covering morphisms}

We introduce a sequence of natural morphisms, the so-called 
{\em cyclic covering morphisms} from $\M_{0,n}$ to moduli spaces of 
{\em unpointed} Deligne-Mumford stable curves of positive genus. The key example is easy to describe:
Given $n$ points $p_{1},\dots, p_{n}$ on $\PP^{1}$, 
for every $p\mid n$ there is a $\mu_p$-cover of 
$\PP^{1}$ totally ramified over points $p_{i}$. 
If $p_{i}=[x_{i}:1]$, this cover is the regular model of the function field extension of $\CC(x)$
defined by $y^{p}=(x-x_{1})\cdots(x-x_{n})$.
The resulting smooth curve has genus $g=(n-2)(p-1)/2$ by the Riemann-Hurwitz formula.
%Algebraically, $C$ is a normalization of $\Spec \bigoplus\limits_{i=1}^{p-1} \O_{\PP^{1}}(-in/p)$,
%where the $\O_{\PP^{1}}$-algebra structure is given by 
%$\O_{\PP^{1}}(-n) \stackrel{\cdot\sum p_{i}}{\longrightarrow} \O_{\PP^{1}}$.
\begin{comment}
Note that $C\ra\PP^{1}$ is indeed totally ramified around $x_{i}$ and is not ramified at $\infty$.
\end{comment}

%Motivated by above, we make the following definition.
\begin{definition}[Cyclic $p$-covering morphism]\label{D:cyclic-smooth}
Let $p \mid n$. We define a regular morphism 
\begin{align*}%\label{E:cyclic-smooth}
f_{n,p}\co \overline{M}_{0,n} \ra \overline{M}_{g}, \quad g=(n-2)(p-1)/2,
\end{align*}
to be a unique morphism that sends $(\PP^1; p_1,\dots, p_n)\in M_{0,n}$
to its degree $p$ cyclic cover totally ramified over $\sum_{i=1}^n p_i$.
%Here, $g$ depends on $n$ and $p$; in particular, $g=(n-2)(p-1)/2$ if 
%\gcd(n,p)=1$.
\end{definition} 
There are two related ways to see that $f_{n,p}$ indeed extends to a morphism $\M_{0,n} \ra \M_g$:
One way is to use the theory of admissible covers of Harris and Mumford \cite{harris-mumford}.
By \cite[p.186]{HM}, there exists a coarse moduli scheme $H$ parameterizing pseudo-admissible $\mu_p$-covers 
(i.e., branched covers where each branching profile at a smooth point is a $p$-cycle in $S_p$). By construction,
$H$ maps finitely to $\M_{0,n}$ and admits a morphism 
to $\M_{g}$. Taking the closure of the section $M_{0,n}\ra H$ given by the cyclic covering trick, and 
using the normality of $\M_{0,n}$, we obtain the necessary extension.
Note that one either has to work with coarse moduli schemes or to take an appropriate root stack of $\overline{M}_{0,n}$ to lift the morphism $f_{n,p}$ to $\Mg{g}$.

We now sketch a related direct construction.
\subsection{The universal $\mu_p$-cover over $\M_{0,n}$}\label{S:universal-p-cover}
As above, consider $p\mid n$. Let $\C\ra \M_{0,n}$
be the universal family of stable $n$-pointed curves and let $\sigman$ be the universal sections.
Observe that $\O_\C(\sum_{i=1}^n\sigma_n)$ is not divisible by $p$ in $\Pic(\C)$: One obstruction to divisibility 
comes from any curve in $\Delta_k$ with $p\nmid k$; another obstruction is due to the fact that 
$\psi_i\in \Pic(\M_{0,n})$ is not divisible by $p$. For any given $1$-parameter family $B\subset \M_{0,n}$, 
this difficulty can be overcome by making a finite base change
of an appropriate order and blowing up the total family at the offending nodes. 
Instead of specifying the base change for every $B$, we indicate a global construction.
\begin{definition} 
\label{D:orbicurve}
Let $p\mid n$. 
We say that $(\X; p_1,\dots, p_n)$ is a {\em $p$-divisible $n$-pointed orbicurve} if
\begin{enumerate}
\item $\X$ is an orbicurve whose coarse moduli space $X$, marked by $p_1,\dots, p_n$, 
is a stable $n$-pointed rational curve.
\item \'{E}tale locally at a node of $X$ of type $\Delta_k$ 
%separating marked points into subsets of sizes $k$ and $n-k$, the 
the orbicurve $\X$ is isomorphic to 
\[
\left[\spec \CC[x,y]/(xy) / \mu_r\right],
\]
where $r=p/\gcd(k,p)$ and $\mu_r$ acts by $(x,y)\mapsto (\alpha^{r-k}, \alpha^k y)$, 
with $\alpha$ a primitive $r^{th}$ root of unity. 
\item \'{E}tale locally over $p_1\in X$, the orbicurve $\X$ is isomorphic to $\left[ \spec \CC[x]/ \mu_p\right]$, where
$\mu_p$ acts by $x\mapsto \beta x$, with $\beta$ a primitive $p^{th}$ root of unity. 
We also require the existence of a section $\sigma_1\co \spec \CC \ra \left[ \spec \CC[x]/ \mu_p\right]$.
\end{enumerate}
\end{definition} 
The notion of a $p$-divisible orbicurve is a mild generalization of 
the notion of even rational orbicurves considered in \cite{AD}. With Definition \ref{D:orbicurve}, for every family 
$(\X\ra B; \sigma_1,\dots, \sigma_n)$ of $p$-divisible orbicurves, 
there is a unique line bundle $\L$ satisfying $\L^{p}\simeq \O_\X(\sum_{i=1}^n \sigma_i)$ and $\sigma_1^*\L\simeq \O_B$.
(In other words, $\O_\X(\sum_{i=1}^n \sigma_i)$ has a unique, up to pullbacks from the base, $p^{th}$ root.)
Consider now the Deligne-Mumford stack $\mathcal{M}$ of $p$-divisible $n$-pointed orbicurves with its 
universal family $\Y\ra \mathcal{M}$. 
We can construct a $\mu_p$-cover $\Z\ra \Y$ simply by applying the cyclic covering construction (see \cite[Proposition 4.1.6]{Laz1}) to the 
divisor $\sum_{i=1}^n \sigma_i$. We omit the details of the construction and refer to \cite{AD} for the special
case of $p=2$. 

Since $\mathcal{M} \ra \M_{0,n}$ is bijective on geometric points, we will 
not distinguish $\Pic(\mathcal{M})\otimes \QQ$ and $\Pic(\M_{0,n})\otimes \QQ$. 
%In particular, 
%any $\QQ$-line-bundle on $\mathcal{M}$ obtained from $\Z\ra\MM$ makes sense as
%a $\QQ$-line-bundle on $\M_{0,n}$. 
We will informally call $\Z\ra \mathcal{M}$, 
the {\em universal stable $\mu_p$-cover} over $\M_{0,n}$, and the fibers of $\Z\ra \mathcal{M}$ will be called
{\em stable $\mu_p$-covers}. The universal stable $\mu_p$-cover induces a morphism $\mathcal{M}\ra\Mg{g}$
that descends to give the desired morphism $f_{n,p}\co \M_{0,n}\ra \M_g$.

% using
%root stacks (for definition and constructions of root stacks see \cite{cadman}). To begin, we take 
%$X_1:=\M_{0,n}\sqrt[p](\psi_1)$ and $X_i:=\M_{0,n}\sqrt[p]\Delta_i$. %(it suffices to take only $i$ not divisible by $p$) Then $X:=X_1\times_{\M_{0,n}}\times X_2\times_{\M_{0,n}}\times \cdots \times X_{N}\times_{\M_{0,n}}$.
%Then $X\ra \M_{0,n}$ is 

\subsection{Divisor classes on $\M_{0,n}$ via cyclic covering morphisms}
\label{S:divisor-classes-lambda}
By pulling back nef divisors on $\Mg{g}$, we obtain 
nef divisors on $\M_{0,n}$. The most interesting, from our point of 
view, divisors are obtained by pulling back: %the tautological divisors on $\Mg{g}$. These include:
the Hodge class $\lambda$, which itself is well-known to be semiample (it defines a morphism from $\M_{g}$ to the 
Satake compactification of $\A_g$), the determinants $\lambda(j)$ of the eigenbundles of the Hodge bundle (these are 
discussed in the sequel), and linear combinations of $\lambda(j)$ and the boundary divisor $\delta_{\irr}$.

\begin{definition}[The Hodge class]\label{D:level1}
Given $p\mid n$, we define 
$\lambda_{n, p}:=f_{n,p}^* \lambda \in \Pic(\M_{0,n})\otimes \QQ$.
\end{definition} 
As we have already observed, 
the divisor class $\lambda_{n, p}$ is semiample and has a simple geometric interpretation: It 
defines a morphism from $\M_{0,n}$ to the Satake compactification of $\A_g$ by sending a rational $n$\nb-pointed 
curve to the abelian part of the generalized Jacobian of its $\mu_{p}$-cover.  Such a geometric interpretation of 
$\lambda_{n, p}$ lents itself to a description of $\lambda_{n, p}^{\perp}\subset N_1(\M_{0,n})$. Namely,
we have the following observation.
\begin{prop}
A curve $B\subset \M_{0,n}$ satisfies $\lambda_{n, p}\cdot B=0$ if and only if every moving component of 
the family of $\mu_p$-covers over $B$ is rational.
\end{prop}
\begin{proof}
Clear from the definition of the Satake compactification.
\end{proof}
Next, we recall a construction of covering families for the boundary divisors $\Delta_k\subset \M_{0,n}$.
\begin{construction}
\label{S:test-curves}
For every $k\in \{3,\dots, \lfloor n/2\rfloor \}$, we consider the 
family of $n$-pointed stable rational 
curves obtained by attaching a constant family of $(n-k+1)$-pointed $\PP^1$ 
along one of the sections to the diagonal of  
$\PP^1\times\PP^1\ra \PP^1$, where $\PP^1\times\PP^1\ra \PP^1$ has $k$ horizontal sections. 
Let $T_{k}\subset\M_{0,n}$ be the
stabilization of this family. We then have 
\begin{align}
\Delta_i\cdot T_k=
\begin{cases} 
2-k & \text{if $i=k$}, \\ 
k    &  \text{if $i=k-1$}, \\
0    & \text{otherwise}. 
\end{cases}
\end{align}
\end{construction}

\begin{corollary}\label{C:degree-0}
\begin{enumerate}
\item[]
\item If $p$ divides one of $a, b, c, d$, then $\lambda_{n, p}$ has degree $0$ on an F-curve $F_{a,b,c,d}$.
\item If $p$ divides $k$, then $\lambda_{n, p}$ has degree $0$ on the curve $T_{k}$ of  
Construction \ref{S:test-curves}.
\end{enumerate}
\end{corollary}
\begin{proof}
The only potential moving component of the $\mu_p$-cover over $F_{a, b, c, d}$ 
is a cyclic cover of the backbone $\PP^1$ branched exactly over 
points $p_A, p_B, p_C, p_D$ (see Section \ref{S:notation} for notation).
Say $p$ divides $a$, then the ramification profile over $p_A$ is trivial. 
It follows that the cover of the backbone does not vary in moduli.

The family over $T_k$ is obtained by gluing two constant families along a section. Since the ramification 
profile over the gluing section is trivial, the resulting cover does not vary in moduli.
\end{proof}

\begin{definition}
\label{D:torelli} 
Let $p\mid n$. Set $g=(n-2)(p-1)/2$. We define
 $\tau_{n,p}\co \M_{0,n} \ra \A^S_g$ to be the composition of the cyclic 
covering morphism $f_{n,p}\co \M_{0,n}\ra \M_g$ and the extended Torelli morphism 
$\tau\co \M_g \ra \A_g^S$ from $\M_g$ to the Satake compactification of $\A_g$. 
\end{definition}
We are now ready to prove the main result of this section (Theorem B from the introduction).
\begin{theorem}\label{T:main-theorem-lambda}
The line bundle $\lambda_{n, p}$ is semiample on $\M_{0,n}$
and defines the morphism 
$$\tau_{n,p}\co \M_{0,n} \ra \A^S_{(n-2)(p-1)/2}$$
to the Satake compactification of $\A_{(n-2)(p-1)/2}$. It contracts the 
boundary divisors $\Delta_{k}$ with $p\mid k$, and its class is 
$$\lambda_{n, p}=\frac{p}{12}\biggl(\left(1-\frac{1}{p^2}\right)\psi-\sum_{k} \left(1-\frac{\gcd(k,p)^2}{p^2}\right)\Delta_k\biggr).
$$
\end{theorem}
\begin{proof}
Set $g=(n-2)(p-1)/2$ and consider the composition 
$\tau_{n,p}\co \M_{0,n} \ra \M_g \ra \A_g^S$. On $\M_g$, we have that
$\lambda=\tau^*(\O_{\A_g^S}(1))$. It follows that $\lambda_{n,p}=f_{n,p}^*(\lambda)$ is semiample.
The class of $\lambda_{n,p}$ is computed in Proposition \ref{P:pullback} below. Finally, 
for every $k$ divisible by $p$, we know from Corollary \ref{C:degree-0} that $\lambda_{n,p}$ has 
degree $0$ on the curve class $T_k$ described in Construction \ref{S:test-curves}. Since deformations of 
$T_k$ cover $\Delta_{k}$, we conclude that the morphism $\tau_{n,p}$ contracts $\Delta_{k}$.
\end{proof}
\begin{remark} When $p$ is prime, we obtain an aesthetically pleasing formula
$$\lambda_{n, p}=\frac{p^2-1}{12p}\biggl(\psi-\Delta+\sum_{p\mid k} \Delta_k\biggr).$$
It is an amusing exercise to verify that $\lambda_{n,p}$ is F-nef. 
\end{remark}

\subsection{A remark on the birational geometry of $\tM{0,n}$}

The effective cone of $\tM{0,n}=\overline{M}_{0,n}/S_n$ has been 
studied already in \cite{keel-mckernan}. Keel and McKernan prove that $\PEff(\tM{0,n})$
is simplicial and is generated by the boundary divisors $\Delta_i$. 
Further, every movable divisor is big. Intuitively, this is explained
by the absence of natural rational fibrations of $\widetilde{M}_{0,n}$ analogous to the forgetful morphisms on 
$\overline{M}_{0,n}$. In particular, (in characteristic $0$) there are no regular contractions of 
$\widetilde{M}_{0,n}$ affecting the interior. 
%otherwise a semiample divisor would be zero on a curve through the interior -- impossibl
%because the divisor is an effective combination of boundary

Subsequently, Hu and Keel asked whether  $\widetilde{M}_{0,n}$ is a Mori dream
space \cite{hu-keel}. If it is, then because every movable divisor is in the interior of  $\PEff(\widetilde{M}_{0,n})$,
Mori chambers cannot share a boundary along a simplicial face of $\PEff(\widetilde{M}_{0,n})$.
It follows that if $\widetilde{M}_{0,n}$ is a Mori dream space, 
then the unique Mori chamber adjacent to the 
face spanned by $\{ \Delta_{j} \mid j\neq i\}$ defines a birational contraction 
$f_i\co \widetilde{M}_{0,n} \dashrightarrow X_i$ such that its exceptional locus 
is $\Exc(f_i)=\bigcup_{j\neq i} \Delta_j$, the set of all boundary divisors but $\Delta_i$.
That a regular contraction $f_2\co \widetilde{M}_{0,n} \ra \Sym^n(\PP^1) \gitq \mathrm{SL}_2$ 
fits the bill for $i=2$ is already observed in \cite{hu-keel}. 
For a general $n$, we are unaware of such $f_i$ for $i\neq 2$.
We note that Rulla computes candidates for the divisors $R_i$ that could be expected to define morphisms $f_i$; 
see \cite[Corollary 6.4]{rulla} and the subsequent remark. However, the question of whether
these divisors are in fact moving and have finitely generated section rings is far from settled. 

For some small $n$, the morphism of Definition \ref{D:torelli} gives us a {\em regular} 
contraction of $\M_{0,n}$ contracting all boundary divisors but one. For example,
the morphism $\tau_{8,2}\co \tM{0,8}\ra \A_3^S$ contracts boundary divisors $\Delta_2$ and $\Delta_4$.
%Whether the (normalization of the) image $\tau_{8,2}(\tM{0,8})$ is a $\QQ$-factorial variety as predicted by %the Mori Dream Space conjecture
%of Hu and Keel 
%\cite{hu-keel}
%is an open question.

\subsection{Cyclic covering morphisms II: Weighted case}

Motivated by the preceding discussion, %of Section \ref{S:divisor-classes-lambda}, 
we generalize the definition of a cyclic covering morphism. 
To this end, we take a sequence of non-negative integers $\vec{d}=(d_1,\dots, d_n)$, called {\em weights} 
and an integer number $p\geq 2$ such that $p\mid \sum_{i=1}^{n} d_i$. 
\begin{definition}\label{D:cyclic-weighted-smooth}
We define a morphism 
$%\begin{align*}%\label{E:cyclic-weighted-smooth}
f_{\vec{d}, p}\co M_{0,n} \ra \M_g,
$ %\end{align*}
by sending a $\PP^1$ marked by $n$ points $p_i=[x_i:1]$
to its degree $p$ cyclic covering totally ramified over $d_1p_1+\cdots+d_n p_n$, 
i.e. the regular model of the function field extension of $\CC(x)$
given by 
$
y^p=(x-x_1)^{d_1}\cdots (x-x_n)^{d_n}.
$
We have that $g=\frac{1}{2}\left[2-2p+\sum_{i=1}^n(p-\gcd(d_i,p)\right]$ by the Riemann-Hurwitz formula.
%, a curve $C\in \M_g$
%isomorphic to $\Spec \bigoplus_{k=0}^{p-1} \L^{k}$, where $\L\simeq 
%\O(\sum_{i=1}^{n}d_ip_i/p)$ and the $\O$-algebra structure on 
%$\bigoplus_{k=0}^{p-1} \L^{k}$ is defined by the isomorphism $\O \simeq \L^{p}$ %is given by the section $\sum_{i=1}^n d_i p_i$.
As before, the morphism extends to $f_{\vec{d},p}\co \M_{0,n}\ra \M_g$.

 We now consider a symmetric variant of the weighted cyclic covering morphism. We define
 \begin{align*}
 f^{S_n}_{\vec{d}, p} = \prod_{\sigma\in S_n} f_{\sigma(\vec{d}), p} \co \M_{0,n} \ra \prod_{\sigma\in S_n} \M_g,
 \end{align*}
Evidently, $f^{S_n}_{\vec{d}, p}$ 
descends to the morphism $f^{S_n}_{\vec{d}, p}\co \tM{0,n} \ra  \prod\limits_{\sigma\in S_n} \M_g$.
 \end{definition}
 
 \begin{definition}[Weighted Hodge class]
 We define 
 \begin{align*}
\lambda_{\vec{d}, p}:=(f^{S_n}_{\vec{d},p})^* \bigotimes_{\sigma \in S_n}  \mathrm{pr}_{\sigma}^*\lambda.
 \end{align*}
 \end{definition}
 
 \begin{prop}\label{P:degree-0}
 If for every partition of $\{1,\dots, n\}$ into subsets %$I_1,I_2,I_3,I_4$ 
 of sizes $a,b,c,d$, the weight 
 of one of the partitions is divisible by $p$,
 then the divisor $\lambda_{\vec{d}, p}$ has degree zero on all F-curves of type $(a,b,c,d)$. 
 %In particular, 
 %$\lambda_{\vec{d}, p}$ descends to a divisor on $\M_{0,n}/\S_n$ having degree $0$ on every F-curve of 
 %type $(a,b,c,d)$.
 \end{prop} 
 \begin{proof}
The proof is the same as in Corollary \ref{C:degree-0}.
 \end{proof}
 
\begin{example}[Divisors on $\Mg{0,7}$] We consider the weight vector $\vec{d}=(1,1,1,1,1,1,0)$. 
 By Proposition \ref{P:degree-0}, the divisor $\lambda_{\vec{d},2}$ has degree $0$ on $F_{2,2,2,1}$. 
 This implies that %up to numerical proportionality 
 $\lambda_{\vec{d},2}\sim \Delta_2+\Delta_3$, which is easily seen to be an extremal ray of $\Nef(\tM{0,7})$. Again by Proposition \ref{P:degree-0}, the divisor $\lambda_{\vec{d},3}$ has degree $0$ on $F_{4,1,1,1}$.
 This implies that %up to numerical proportionality 
 $\lambda_{\vec{d},3}\sim \Delta_2+3\Delta_3$, which is another extremal ray of $\Nef(\tM{0,7})$. Thus $\lambda_{\vec{d},2}$
 and $\lambda_{\vec{d},3}$ account for all extremal rays of $\Nef(\tM{0,7})$.
 \end{example}

\section{Divisor classes associated to families of stable cyclic covers}
\label{S:divisor-classes}

In this section, we prove our main technical intersection-theoretic results.

\begin{prop}[Pullback formulae for cyclic covering morphisms]
\label{P:pullback}
Let $\vec{d}=(d_1,\dots, d_n)$. For any $p\mid \sum_{i=1}^n d_i$, consider the cyclic covering morphism
$f_{\vec{d},p}\co \M_{0,n} \ra \M_g$ of Definition \ref{D:cyclic-weighted-smooth}. For every boundary divisor 
$\Delta_{I,J}$, set $d(I)=\sum_{i\in I} d_i$. Then the divisor classes
$\lambda, \delta_{\irr}=\delta_0$, and $\delta_{\red}=\delta_1+\cdots+\delta_{\lfloor g/2\rfloor}$ pullback 
as follows:
\begin{eqnarray}\label{E:weighted-pullback}
\begin{aligned}
f_{\vec{d},p}^*(\lambda) &=\frac{1}{12p} \left[ \sum_{i=1}^n \left(p^2-\gcd(d_i,p)^2\right) \psi_i-\sum_{I,J} \left(p^2-\gcd(d(I),p)^2\right)\Delta_{I,J}\right]  \\ 
f_{\vec{d},p}^*(\delta_{\irr}) &= \frac{1}{p}\sum_{\quad I,J: \ \gcd(d(I), p)>1} \gcd(d(I),p)^2\Delta_{I,J} \\
f_{\vec{d},p}^*(\delta_{\red}) &= \frac{1}{p}\sum_{\quad I,J: \ \gcd(d(I), p)=1}\Delta_{I,J}.
\end{aligned}
\end{eqnarray}
In particular, in the unweighted case we have:
\begin{eqnarray}\label{E:pullback}
\begin{aligned}
f_{n,p}^*(\lambda) &=\frac{p}{12}\left(\left(1-\frac{1}{p^2}\right)\psi-\sum_{k} \left(1-\frac{\gcd(k,p)^2}{p^2}\right)\Delta_k\right), \\ 
f_{n,p}^*(\delta_{\irr}) &= \frac{1}{p}\sum_{k: \ \gcd(k,p)>1} \gcd(k,p)^2\Delta_{k}, \\
f_{n,p}^*(\delta_{\red}) &=  \frac{1}{p}\sum_{k: \ \gcd(k,p)=1} \Delta_{k}.
\end{aligned}
\end{eqnarray}
\end{prop}

\begin{proof} 
%Note that divisors $\lambda_{n, p}$ and $\psi-\Delta+\sum_{p\vert k} \Delta_k$
%have degree zero on all F-curves $F_{a,b,c, d}$ with at least one of $a, b, c, d$ divisible by $p$. This however does %not even suffice to prove that $\lambda_{n, p}$ is a multiple of $\psi-\Delta+\sum_{p\vert k} \Delta_k$.

Due to the inductive nature of the boundary of $\M_{0,n}$,
it suffices to prove Formulae \eqref{E:weighted-pullback}
for generically smooth families of a special form. Namely, we consider a family $\X' \ra B$
of stable $n$-pointed rational curves with $B\simeq \PP^1$ and with $\X'$ constructed as follows:
Let $\X\ra B$ be a $\PP^1$-bundle with a section $\Sigma_1$ of negative self-intersection $\Sigma_1^2=-r$ 
and $(n-1)$ sections $\{\Sigma_i\}_{i=2}^n$ disjoint from $\Sigma_1$ and in general position. Then we take 
$\X'$ to 
be simply the stabilization of $\X\ra B$, i.e. the ordinary blow-up of $\X$ at the points where 
sections $\{\Sigma_i\}_{i=2}^n$ intersect. We also let $S=\{t_1,\dots, t_m\}$ be 
the points in $B$ over which sections $\{\Sigma_i\}_{i=2}^n$ intersect. (Evidently, $m=\binom{n-1}{2}r$ but 
we will not need this in the sequel.)

The family $\X'\ra B$ induces a map $B\ra \M_{0,n}$. We now follow what happens to $B$ under the cyclic 
$p$-covering morphism $f_{n,p}\co \M_{0,n}\ra \M_g$. 

Observe that $D:=\sum_{i=1}^n d_i \Sigma_i$ will be divisible by $p$ in $\Pic(\X)$ as long as $p\mid r$. 
The family of stable $\mu_p$-covers over $B$ is constructed out of $(\X\ra B; \{\Sigma_i\}_{i=1}^n)$
in the following steps:
\begin{enumerate}
\item[] (Step 1) Apply the cyclic covering trick to construct a $\mu_p$-cover ramified over $D$. 
Obtain the family $\Y\ra B$ of generically smooth $\mu_p$-covers over $B$, with singular fibers over $S$.
\item[] (Step 2)  Consider a finite base change $B'\ra B$ of degree $p$ totally ramified over points in $S$. 
Denote $\Y':=\Y\times_B B'$.
\item[](Step 3)  Perform weighted blow-ups on $\Y'$ to arrive at the stable family $\Z\ra B'$.
\end{enumerate}
We proceed to elaborate on the last step.
%We now describe the stable reduction of $\Y$. 
The problem is local and so we work locally
around a point where two sections $\Sigma_i$ and $\Sigma_j$ meet. 
The local equation of $\Y'$ at this point is $y^p=(x-at^p)^{d_i}(x-bt^p)^{d_j}$, 
where $t$ is the uniformizer on $B'$. 

Set $q=\gcd(p, d_i+d_j)$. A weighted blow-up 
with weights $w(x,y,t)=(p/q,(d_i+d_j)/q,1)$ followed by the normalization replaces the singularity 
$y^p=(x-at^p)^{d_i}(x-bt^p)^{d_j}$ by a smooth curve $G_{ij}$ which now meets the rest of the fiber in $q$ points. 
The self-intersection of $G_{ij}$ in $\Z$ is $(-1)$. Observe that at  
each of the points where $G_{ij}$ meets the rest of the fiber, $\Z$ has an $A_{q-1}$ singularity. 
It follows that this singular fiber contributes
$q^2$ to $\delta_{\irr}\cdot B'$ if $q>1$, and contributes $1$ to $\delta_{\red}\cdot B'$ if $q=1$.  
 Summarizing:
 \begin{align}\label{E:delta}
 \delta_{\irr}\cdot B'=\sum_{\gcd(d_i+d_j,p)>1} \gcd(d_i+d_j,p)^2\Delta_{ij}, \quad
 \delta_{\red}\cdot B'=\sum_{\gcd(d_i+d_j, p)=1}\Delta_{ij}.
 \end{align}

We proceed to compute the remaining numerical invariants of the family $\Z\ra B'$. First, 
we set $r_i=p/\gcd(d_i,p)$ and consider the branch divisor 
$$
\Br=\sum_{i=1}^n\left(\frac{r_i-1}{r_i}\right)\Sigma_i.
$$
If $\pi\co \Y\ra \X$ is the cyclic cover constructed in (Step 1), 
then $\omega_{\Y/B}=\pi^*(\omega_{\X/B}+\Br)$. Since ${\Y'/B'}$ is obtained from $\Y/B$ by 
a finite base change of degree $p$, we conclude that 
\begin{align}\label{E:base-change}
\omega^2_{\Y'/B'}=p\omega^2_{\Y/B}=p^2(\omega_{\X/B}+\Br)^2.
\end{align}

Next, if we let $\xi\co \Z\ra \Y'$ be the composition of the weighted blow-ups in (Step 3), 
then the exceptional divisors of $\xi$ are precisely curves $G_{ij}$ described above.
%Since $\xi$ is a weighted blow-up with weights $(p/q,(d_i+d_j)/q,1)$ 
We have $\omega_{\Z/B'}=\xi^*(\omega_{\Y'/B'})+\sum_{i<j} a_{ij} G_{ij}$ and our immediate goal is to compute $a_{ij}$.
Observe that each exceptional divisor $G_{ij}$
is a degree $p$ cover of $\PP^1$ with $3$ branch points and the 
branching profile $(\tau^{-d_i-d_j},\tau^{d_i},\tau^{d_j})$ where
$\tau$ is a $p$-cycle in $S_p$. Therefore, by the Riemann-Hurwitz formula, we have
$$
2g(G_{ij})+2p-2=(r_i-1)p/r_i+(r_i-1)p/r_j+(p/q-1)q
$$ 
and so $2g(G_{ij})-2=p-q-p/r_i-p/r_j$ (as before, $q=\gcd(d_i+d_j,p)$). 
To determine $a_{ij}$ we apply adjunction: because the singularities of $\Z$ are Du Val, we have 
$\omega_{\Z/B'}\cdot G_{ij}=\deg \omega_{G_{ij}}+q=p-p/r_i-p/r_j$. Recalling that $G_{ij}^2=-1$, we obtain  
$a_{ij}=p-p/r_i-p/r_j$. It follows that 
\begin{align}\label{E:kappa-Z-Y}
\omega^2_{\Z/B'}=\omega^2_{\Y'/B'}-p^2\sum_{i<j}(1-1/r_i-1/r_j)^2\Delta_{ij}.
\end{align}

Recall that the family $\X'\ra B$ of stable $n$-pointed rational curves associated to 
$\X\ra B$ is obtained by blowing up points where sections $\{\Sigma_i\}_{i=1}^n$ intersect. It follows that 
\begin{align*}
\Br^2 &=-\sum_{i=1}^n \left(\frac{r_i-1}{r_i}\right)^2 \psi_i+\sum_{i<j} \left(2-\frac{1}{r_i}-\frac{1}{r_j}\right)^2\Delta_{ij}, \\
\omega_{\X/B}\cdot \Br &=\sum_{i=1}^n \frac{r_i-1}{r_i}\psi_i-\sum_{i<j}\left(2-\frac{1}{r_i}-\frac{1}{r_j}\right)\Delta_{ij}.
\end{align*}
Using $\omega^2_{\Y'/B'}=p^2(\omega^2_{\X/B}+2\omega_{\X/B}\Br+\Br^2)$ from Equation \eqref{E:base-change}, we compute 
\begin{align}\label{E:kappa-Y}
\omega^2_{\Y'/B'}=p^2\left[\sum_{i=1}^n \frac{r^2_i-1}{r^2_i}\psi_i\right] - p^2\left[\sum_{i<j}\left(2-\frac{1}{r_i}-\frac{1}{r_j}\right)\left(\frac{1}{r_i}+\frac{1}{r_j}\right) \Delta_{ij}\right]
\end{align}

We finally compute the degree of $\lambda$ on the family $\Z\ra B'$. Combining Equations \eqref{E:delta}, \eqref{E:kappa-Z-Y}, \eqref{E:kappa-Y}, and   
Mumford's formula $12\lambda=\kappa+\delta$, we obtain:
\begin{multline*}
 \lambda=\frac{1}{12} \biggl[ p^2\sum_{i=1}^n \frac{r^2_i-1}{r^2_i}\psi_i -  p^2 \sum_{i<j}\left(2-\frac{1}{r_i} 
 -\frac{1}{r_j}\right)\left(\frac{1}{r_i}+\frac{1}{r_j}\right) \Delta_{ij} \\-p^2\sum_{i<j}(1-1/r_i-1/r_j)^2\Delta_{ij}+\sum_{i<j} \gcd(p, d_i+d_j)^2\Delta_{ij} \biggr] \\
=\frac{p^2}{12} \left[ \sum_{i=1}^n \left(1-\frac{\gcd(d_i,p)^2}{p^2}\right) \psi_i-\sum_{i<j} \left(1-\frac{\gcd(p,d_i+d_j)^2}{p^2}\right)\Delta_{ij}\right]
 \end{multline*}
\end{proof}

\section{Eigenbundles of Hodge bundles over $\M_{0,n}$}
\label{S:eigen-hodge-classes}

We continue the study of cyclic covering morphisms $f_{n,p}\co \M_{0,n}\ra \M_g$, defined for every
$p\mid n$ and $g=(n-2)(p-1)/2$. In Section \ref{S:divisor-classes}, we have studied the Hodge class $\lambda_{n,p}$,
which by Definition \ref{D:level1} is the determinant of $\EE$ -- the pullback 
of the Hodge bundle from $\Mg{g}$ via $f_{n,p}$. We now turn our attention to $\EE$ itself. 

By construction, $\EE$ is the Hodge bundle associated   
to the family of stable $\mu_{p}$-covers over $\M_{0,n}$ constructed in Section \ref{S:universal-p-cover}. 
The presence of the $\mu_{p}$-action urges us to consider  
the eigenbundle decomposition of $\EE$.  
We identify characters of $\mu_p$ with integers $\{0,1,\dots, p-1\}$. Let $\alpha$ be a generator of $\mu_p$.
For every character $j$, and every stable cyclic $\mu_p$-cover $C$, we define 
$$
\HH^{0}(C, \omega_{C})_j:=\{ \omega\in \HH^0(C,\omega_C) \mid \alpha\cdot \omega=\alpha^j\omega \}.
$$
We refer to the 1-forms in $\HH^{0}(C, \omega_{C})_{j}$ as {\em forms of weight $j$}. 
Since $\HH^0(C,\omega_C)_0=(0)$, the eigenbundle decomposition of $\EE$ with respect 
to the $\mu_p$-action is
\begin{align*}
\EE=\bigoplus_{j=1}^{p-1}\EE_{j}.
\end{align*}
%Here, we let $\EE_{j}$ to be the bundle on $\M_{0,n}$ whose fiber over $(P; \sigman)\in \M_{0,n}$ 
%is defined as follows: If $C$ is the stable $\mu_{p}$-cover of $(P; \sigman)$ 
%$$y^{p}=(x-x_{1})\cdots(x-x_{n}).$$ 
%then the fiber of $\EE_{j}$ at $(P; \sigman)$ is $\HH^{0}(C, \omega_{C})_{j}$.
\subsection{Determinants of eigenbundles $\EE_j$}
Since the Hodge bundle $\EE$ is semipositive \cite{kollar-projectivity}, the eigenbundles $\EE_j$ are semipositive
as well. Hence, their determinants are nef divisors on $\M_{0,n}$. 
At the first sight, the task of computing the determinant of $\EE_{j}$ appears daunting. However, there is one situation 
where this can be done explicitly. Namely, we restrict ourselves to the family of stable 
$\mu_{p}$-covers over an F-curve of type $(a,b,c,d)$. Then the moving component of the stable 
$\mu_{p}$-cover in question is much studied family of cyclic covers defined by the equation
\begin{align}\label{family}
C_{\lambda}: \quad y^{p}=x^{a}(x-1)^{b}(x-\lambda)^{c}, \quad \lambda\in \PP^1. 
\end{align}
Still how does one compute the degree of $\EE_{j}$ on this family? 
Before we answer this question completely,
we describe a situation when the computation is made without any effort. 

\noindent
Observe that 
a rational 1-form of weight $j$ on $C_\lambda$ 
is necessarily of the form $y^{j}dx/f(x)$. Assuming for simplicity that $p$ is coprime to 
$a$, $b$, $c$, and $a+b+c$, we see that $(dx)=(p-1)[0]+(p-1)[1]+(p-1)[\lambda]-(p+1)[\infty]$ and 
$(y)=a[0]+b[1]+c[\lambda]-(a+b+c)[\infty]$. Evidently, $y^{j}dx/f(x)$ can be regular
 if and only if there is an effective integer linear combination of vectors 
 $$(p-1,p-1,p-1,-(p+1)), (ja,jb,jc,-j(a+b+c)), (p,0,0,-p), (0,p,0,-p), (0,0,p,-p).$$ %with non-negative entries
To see whether an effective linear combination exists, 
we make the first three entries (corresponding to orders of vanishing at $0,1,\lambda$) as small as possible.
This is clearly achieved by the vector 
\begin{align*}
(\modp{aj-1}{p},\modp{bj-1}{p},\modp{cj-1}{p},  2p-4-\modp{aj-1}{p}-\modp{bj-1}{p}-\modp{cj-1}{p})
\end{align*}
It follows that if $\modp{aj-1}{p}+\modp{bj-1}{p}+\modp{cj-1}{p}\geq 2p-3$, then 
there are no forms of weight $j$ on the generic $C_{\lambda}$, i.e. 
$\HH^0(C_\lambda, \omega_{C_\lambda})_j=(0)$ for the generic $\lambda$. Thus $\det \EE_j\cdot F_{a,b,c,d}=0$.
We summarize the discussion so far in the following proposition.
\begin{prop}\label{P:determinants-prelim}
If $\modp{aj}{p}+\modp{bj}{p}+\modp{cj}{p}+\modp{dj}{p}=3p$, then 
$$\det \EE_j\cdot F_{a,b,c,d}=0.$$
\end{prop}
\begin{proof}
By Lemma \ref{L:weight-j-4-branched}, we have $\HH^0(C,\omega_C)_j=(0)$ for the generic curve in Family \eqref{family}. 
\end{proof}

We now proceed to generalize this observation. We note that 
the results of the following proposition are not new and can be found in 
\cite{bouw-moller} and \cite{eskin-kontsevich-zorich}. 
For completeness, we include proofs 
in Section \ref{S:appendix}. We use the notation of Construction \ref{construction}.
\begin{prop}\label{P:determinants-main}
Let $\EE$ be the Hodge bundle of the universal cyclic $\mu_p$-cover of type $(a,b,c,d)$ over $\M_{0,4}$. 
Then:
\begin{enumerate}
%\item The eigenbundle $\EE_j$ has rank $0$ for all $j\in\{0,1,\dots,p-1\}$ such that 
 %\begin{align*}%\label{E:p}
 %\modp{aj}{p}+\modp{bj}{p}+\modp{cj}{p}+\modp{dj}{p}=3p.
 %\end{align*}
\item The eigenbundle $\EE_j$ has rank $2$ and $c_1(\EE_j)=0$ for all $j\in\{0,1,\dots,p-1\}$  such that 
 \begin{align*}%\label{E:3p}
 \modp{aj}{p}+\modp{bj}{p}+\modp{cj}{p}+\modp{dj}{p}=p.
 \end{align*}
\item The eigenbundle $\EE_j$ has rank $1$ for all $j\in\{0,1,\dots,p-1\}$ such that
 \begin{align*}%\label{E:2p}
 \modp{aj}{p}+\modp{bj}{p}+\modp{cj}{p}+\modp{dj}{p}=2p.
 \end{align*}
Moreover, 
\begin{align*}
\deg \EE_j=\frac{1}{p}\min\{\modp{aj}{p},\modp{bj}{p},\modp{cj}{p},\modp{dj}{p}, 
\modp{-aj}{p},\modp{-bj}{p},\modp{-cj}{p},\modp{-dj}{p}\}.
\end{align*}
\end{enumerate}
\end{prop}
\begin{proof}
This is Proposition \ref{P:determinants} from Section \ref{S:appendix}.
\end{proof}
Using Propositions \ref{P:determinants-prelim} and \ref{P:determinants-main}, we now prove Theorem A from the introduction.
We begin with a preliminary lemma.
\begin{lemma}\label{lemma}
 Suppose $p\mid n$.
Regard an F-curve of type $(a,b,c,d)$ as a map $\iota\co \M_{0,4}\ra \M_{0,n}$. 
Let $\EE$ be the pullback to $\M_{0,n}$ of the Hodge bundle under the cyclic covering morphism 
$f_{n,p}\co \M_{0,n} \ra \M_{(n-2)(p-1)/2}$ and let $\FF$ be the pullback to $\M_{0,4}$ 
of the Hodge bundle under the weighted cyclic covering morphism $f_{(a,b,c,d), p}\co \M_{0,4}\ra \M_h$.
Then for every character $j$ of $\mu_p$, we have
$$
c_1(\iota^*\EE_j)=c_1(\FF_j).
$$
\end{lemma}
\begin{proof}
In the situation of the lemma, 
$\EE$ is an extension of $\FF$ by a trivial vector bundle on $\M_{0,4}$. The statement follows.
% (cf. \cite{moduli-of-curves}).
\end{proof}

\begin{theorem}\label{T:theorem-A} Let $\EE$ be the pullback to $\M_{0,n}$ 
of the Hodge bundle over $\Mg{g}$ via the cyclic $p$-covering morphism $f_{n,p}$. Let
$\EE_j$ be the eigenbundle of $\EE$ associated to the character $j$ of $\mu_p$. Then
the eigenbundle $\EE_j$ is semipositive. The determinant line bundle $\lambda_{n,p}(j):=\det \EE_j$ is nef and
$$\lambda_{n,p}(j)= \frac{1}{p}\DD^1_{n, jn/p},$$
where $\DD^1_{n, jn/p}$ is the symmetric $\sl_n$ level $1$ conformal blocks divisor associated to the 
fundamental weight $jn/p$.
\end{theorem}

\begin{proof}
By Koll\'{a}r's semipositivity results \cite[Theorem 4.3 and Remark 4.4]{kollar-projectivity}, the Hodge
bundle $\EE$ over $\Mg{g}$ 
is semipositive in characteristic $0$. Thus every eigenbundle $\EE_j$ is semipositive. We conclude 
that $\det(\EE_j)$ is nef. Finally, by Propositions \ref{P:determinants-prelim}--\ref{P:determinants-main}, Lemma \ref{lemma},
and \cite[Proposition 5.2]{fakh}
the degrees of $p\lambda_{n,p}(j)$ and $\DD^1_{n, jn/p}$ are equal on every 
F-curve. 
\end{proof}
The following result shows that by considering weighted cyclic covering morphisms 
every $\sl_p$ level $1$ conformal blocks line bundle on $\M_{0,n}$ arises 
as the determinant of the eigenbundle corresponding to the character $j=1$, after a suitable choice of a weighted
covering morphism.
\begin{theorem}\label{T:theorem-AA} 
For a weight vector $\vec{d}=(d_1,\dots, d_n)$, let $p$ be an integer dividing $\sum_{i=1}^n d_i$.
Denote by $\EE$ the pullback to $\M_{0,n}$ 
of the Hodge bundle over $\Mg{g}$ via the weighted cyclic $p$-covering morphism $f_{\vec{d},p}$.
Let $\EE_1$ be the eigenbundle of $\EE$ associated to the character $j=1$ of $\mu_p$. Then
$\EE_1$ is semipositive and its determinant $\lambda_{\vec{d},p}(1):=\det \EE_1$ is nef. Moreover, 
$$\lambda_{\vec{d},p}(1)=\frac{1}{p} \DD(\sl_p, 1, (w_{d_1},\dots, w_{d_n})),$$
where $\DD(\sl_p, 1, (w_{d_1},\dots, w_{d_n}))$ is the $\sl_p$ level $1$ conformal blocks divisor associated to the 
sequence $(w_{d_1},\dots, w_{d_n})$ of fundamental weights.
\end{theorem}
\begin{proof}
%Compare formulae of Propositions \ref{P:determinants-prelim}--\ref{P:determinants-main}
%and \cite[Proposition 5.2]{fakh}.
The proof is the same as that of Theorem \ref{T:theorem-A}.
\end{proof}

As a corollary of Theorem \ref{T:main-theorem-lambda} and Theorem \ref{T:theorem-A}, 
we obtain the following result.
\begin{prop} Let $\EE$ be the pullback of the Hodge bundle over $\Mg{g}$ to $\M_{0,n}$ via
the cyclic $p$-covering morphism $f_{n,p}$. Then for every $j=1,\dots, p-1$, the morphism associated to the
semiample line bundle $\lambda_{n,p}(j)\sim \DD^1_{n,jn/p}$ contracts boundary 
divisors $\Delta_k$ with $p\mid k$. 
\end{prop}
\begin{proof} Each line bundle $\lambda_{n,p}(j)$ is semiample because by Theorem \ref{T:theorem-A} 
it is a multiple of a 
conformal blocks line bundle $\DD^1_{n,jn/p}$, which is generated by global sections \cite[Lemma 2.5]{fakh}. 
From the eigenbundle decomposition $\EE=\bigoplus_{j=1}^{p-1} \EE_j$, we deduce that $f^*_{n,p}(\lambda)$
is an effective combination of $\lambda_{n,p}(j)$ for $j=1,\dots, p-1$. Since the morphism associated to 
$f^*_{n,p}(\lambda)$ contracts all $\Delta_k$ with $p\mid k$ by Theorem \ref{T:main-theorem-lambda}, 
the same holds for each 
$\lambda_{n,p}(j)$.
\end{proof}

 \begin{example}[The $\psi$-class] Take $\vec{d}=(n-2,2,1,1,\dots, 1)$ and $p=n-1$. Consider the weighted cyclic covering morphism $f_{\vec{d},p}\co \M_{0,n}\ra \M_g$
 and let $\EE$ be the pullback of the Hodge bundle from $\M_g$. Let $\lambda_{\vec{d},p}(1)$ 
 be the determinant of the eigenbundle $\EE_1$ corresponding to the character $j=1$ of $\mu_{p}$.
 Then 
 $$
\psi \sim \sum_{\sigma\in S_n} \lambda_{\sigma(\vec{d}),p}(1).
 $$
 \end{example}

We finish this section by giving a new formula for the classes of line bundles $\lambda_{\vec{d},p}(j)$ and,
therefore, for all $\sl_p$ level $1$ conformal blocks divisors.
As will be clear from the proof of the proposition, the novelty here is 
in finding an expression for the divisor classes in question that behaves nicely under 
restriction to the boundary. 
%We will use this formula to construct new extremal rays of $\Nef(\tM{0,n})$ in 
%Section \ref{S:new}.
\begin{prop}\label{P:new-formula}
For a weight vector $\vec{d}=(d_1,\dots, d_n)$, let $p$ be an integer dividing $\sum_{i=1}^n d_i$.
Let $\EE$ be the pullback to $\M_{0,n}$ of the Hodge bundle over $\Mg{g}$ via the weighted cyclic $p$-covering morphism $f_{\vec{d},p}$ and let $\EE_j$ be the eigenbundle of $\EE$ associated to the character $j$ of $\mu_p$. 
Then\footnote{As before, we denote by $\modp{a}{p}$ the representative in $\{0,1,\dots, p-1\}$ of the residue of $a$ modulo $p$.}
\[
\lambda_{\vec{d},p}(j):=\det \EE_j =\frac{1}{2p^2}\left[\sum_{i=1}^n \modp{jd_i}{p}\modp{p-jd_i}{p}\psi_i
-\sum_{I,J} \modp{jd(I)}{p}\modp{jd(J)}{p}\Delta_{I,J}\right].
\]
\end{prop}
\begin{remark}
To get a formula for $\DD(\sl_p, 1, (d_1,\dots,d_n))$, take $j=1$ and multiply by $p$.
\end{remark}
\begin{proof} 
Set $\D(d_1,\dots,d_n):=\sum_{i=1}^n \modp{jd_i}{p}\modp{p-jd_i}{p}\psi_i
-\sum_{I,J} \modp{jd(I)}{p}\modp{jd(J)}{p}\Delta_{I,J}$. 
We note immediately that both $\EE_j$ and 
$\D(d_1,\dots,d_n)$ behave functorially under restriction to the boundary. Namely, if $I=(i_1,\dots, i_k)$ and
$J=(j_1, \dots, j_{n-k})$ and $B\subset \Delta_{I,J} \subset \M_{0,n}$ is a family of generically reducible curves
obtained by gluing families $B_1\subset \M_{0,k+1}$ and $B_2\subset \M_{0,n-k+1}$, then
\[
\D(d_1,\dots,d_n)\cdot B=D(d_{i_1},\dots, d_{i_k}, \sum_{r\in J} d_r)\cdot B_1+
\D( \sum_{r\in I} d_r, d_{j_1},\dots, d_{j_{n-k}})\cdot B_2.
\]
The same holds for $\lambda_{\vec{d},p}(j)$. 
Therefore, it suffices to show that the degrees of two divisor classes are the same on all F-curves.

Consider the F-curve $F:=F_{I_1,I_2,I_3,I_4}$. The moving family over any
 F-curve has exactly $4$ sections of self-intersection $(-1)$ and exactly $3$ nodal fibers.
  Denote 
  $$a:=\modp{jd(I_1)}{p}, \ b:=\modp{jd(I_2)}{p}, \ c:=\modp{jd(I_3)}{p},  \
d:=\modp{jd(I_4)}{p},$$ and suppose without loss of generality that $a\leq b\leq c \leq d$.
Since both $\D$ and $\lambda_{\vec{d},p}(j)$ are invariant under substitution 
$j\mapsto p-j$, it suffices to treat the following cases:

{\em Case 1: $a+b+c+d=2p$.} In this case the degree of $\D$ on $F$ is 
\begin{multline*}
a(p-a)+b(p-b)+c(p-c)+d(p-d)\\ -(a+b)(p-a-b)-(a+c)(p-a-c)-(a+d)(p-a-d) =2pa,
% -A^2-B^2-C^2-D^2+(A^2+2AB+B^2)+(A^2+2AC+C^2)+(A^2+2AD+D^2)+p(A+B+C+D-A-B-A-C-A-D)
%=2A^2+2AB+2AC+2AD-2pA=2A(2p)-2Ap=2pA$
\end{multline*}
if $a+d\leq p$, and is 
\begin{multline*}
a(p-a)+b(p-b)+c(p-c)+d(p-d)\\ -(a+b)(p-a-b)-(a+c)(p-a-c)-(b+c)(p-b-c) =2p(p-d),
\end{multline*}
if $a+d\geq p$.

{\em Case 2: $a+b+c+d=p$.} In this case the degree of $\D$ on $F$ is 
\begin{multline*}
a(p-a)+b(p-b)+c(p-c)+d(p-d)\\ -(a+b)(p-a-b)-(a+c)(p-a-c)-(a+d)(p-a-d) =0.
\end{multline*}
By comparing these formulae with the formulae of Proposition \ref{P:determinants} we conclude the proof.

\end{proof}

\section{New extremal rays of $\Nef(\tM{0,n})$} %Beyond conformal blocks
\label{S:new}

In this section, we prove  Theorem C from the introduction and thus 
construct new extremal rays of $\Nef(\tM{0,n})$. 
Our approach is directly via the intersection theory for one-parameter families of curves. 
%In particular, we do not use the Contraction Theorem (cf. \cite{keel-mckernan, FG}).
% that are not determinants of conformal blocks bundles.
The key ingredient of our construction is the following result due to Stankova:
\begin{prop}[\text{\cite[Theorem 7.3]{stankova}}]
\label{L:stankova}
 If $\X\ra C$ is a family of generically
smooth trigonal curves of genus $g$, then $(\delta_{\irr}\cdot C)/(\lambda\cdot C)\leq 36(g+1)/(5g+1)$.
\end{prop} 
\begin{proof}
In fact, Stankova proves a stronger inequality
$$(\delta \cdot C) \leq \frac{36(g+1)}{(5g+1)} (\lambda\cdot C).$$
The statement now follows 
from $(\delta \cdot C)=(\delta_{\irr}\cdot C)+(\delta_{\red} \cdot C)\geq (\delta_{\irr}\cdot C)$ for any generically
smooth family of stable curves.
\end{proof}
\begin{theorem}\label{T:3}
Suppose $3\mid n$. Consider the cyclic covering morphism $f_{n,3}\co \M_{0,n}\ra \M_{n-2}$. The line bundle  
$$f^*_{n,3}(9\lambda-\delta_{\irr})=2\psi-2\Delta-\sum_{3\mid k}\Delta_{k}$$
generates an extremal ray of $\tM{0,n}$.
\end{theorem}
\begin{proof}
\noindent
We compute the divisor class using Proposition \ref{P:pullback} to obtain
$$
f^*_{n,3}(9\lambda-\delta_{\irr})=2(\psi-\Delta+\sum_{3\mid k}\Delta_{k})-3\sum_{3\mid k}\Delta_{k}
=2\psi-2\Delta-\sum_{3\mid k}\Delta_{k}.
$$

\begin{proof}[Proof of nefness:]
Does writing divisor $2\psi-2\Delta-\sum_{3\mid k}\Delta_{k}$ as a pullback of $9\lambda-\delta_{\irr}$ from 
$\M_{n-2}$ help to establish 
its nefness on $\tM{0,n}$? By now even a casual reader will guess that the answer is yes. 
%We begin by observing that 
%$9\lambda-\delta_{\irr}$ is superadditive under normalization of a generically reducible family:
Suppose that a family $\X\ra B$ of stable $\mu_3$-covers is generically reducible. Then we can write
$\X=\X_{1}\cup \X_{2}$, where $\X_{i}\ra B$ are themselves families of stable $\mu_3$-covers, 
and where the union is formed by identifying sections. 
We have then the inequality
$$(9\lambda-\delta_{\irr})_{\X/B}\geq 
 (9\lambda-\delta_{\irr})_{\X_{1}/B}+(9\lambda-\delta_{\irr})_{\X_{2}/B}.$$
Therefore, it suffices to show that $9\lambda-\delta_{\irr}$ is non-negative
on every family of generically smooth cyclic triple covers. 
Clearly, the genus of a $\mu_3$-cover is at least $2$.
If the genus is $3$ or more, then we are done by the inequality 
$
\dfrac{36(g+1)}{5g+1}\lambda-\delta_{\irr}\geq 0
$ of Proposition \ref{L:stankova}.
It remains to treat the genus $2$ case. 
The generic $\mu_3$-cover of genus $2$ is a degree $3$ cover of $\PP^{1}$ branched over $4$ points 
with ramification profile $(\tau^{2},\tau^{2},\tau, \tau)$ where $\tau$ is a $3$-cycle in $S_{3}$. A family 
of genus $2$ stable $\mu_3$-covers is obtained by varying the cross-ratio of $4$ points on $\PP^1$.
But the divisor $f^*_{n, 3}(9\lambda-\delta_{\irr})$ is zero on such a family by an explicit computation, or by 
observing that  $f_{n,3}^*(9\lambda-\delta_{\irr})=2\psi-2\Delta-\sum_{3\mid k}\Delta_{k}$ is zero 
on any F-curve congruent to $(2,2,1,1)$ modulo $3$ (and positive on any other F-curve)!
\end{proof}

\begin{proof}[Proof of extremality:] 
Since the Picard number of $\tM{0,n}$ is $\lfloor n/2\rfloor-1$, a nef divisor 
generates an extremal ray of $\Nef(\tM{0,n})$ 
if it intersects trivially $\lfloor n/2\rfloor-2$ linearly independent effective curves in $N_1(\tM{0,n})$. 
We follow the approach of \cite{agss} and look for 
$\lfloor n/2\rfloor-2$ linearly independent F-curves which $2\psi-2\Delta-\sum_{3\mid k}\Delta_{k}$ intersects
trivially. By above, we have to consider F-curves congruent to $(2,2,1,1)$ modulo $3$.

We treat the case of $n\equiv 0 \mod 12$ in full detail and indicate the necessary modifications
for the remaining cases.

\noindent
{\em Case of $n=12t$:}
It is easy to verify that for $n=12$ the F-curves $F_{5,5,1,1}, F_{4,4,2,2}, F_{1,2,2,7}, F_{1,1,2,8}$ 
are linearly independent. Suppose now $t\geq 2$.
We let 
\begin{multline*}
\N_i=\{F_{6i+1,1,2,n-4-6i}, F_{6i+1,2,2,n-5-6i}, F_{6i+2,1,1,n-4-6i}, \\
F_{6i+5,1,1,n-7-6i}, F_{6i+4,2,2,n-8-6i}, F_{6i+5,1,2,n-8-6i}\} 
\end{multline*}
for $i=0,\dots, t-2$,
and let $\N_{t-1}=\{F_{6t-5,1, 2, 6t+2}, F_{6t-4,1,1, 6t+2}, F_{6t-5,2,2, 6t+1},  F_{6t-1,1,1, 6t-1}\}$.
Then for $1\leq i\leq t-2$, the intersection pairing of $\N_i$ with the divisors 
$\{\Delta_{6i+k}\}_{k=3}^{8}$  %, \Delta_{6i+4}, \Delta_{6i+5}, \Delta_{6i+6}, \Delta_{6i+7}, \Delta_{6i+8}$ 
is the following:
\begin{align*}
\begin{array}{l | rrrrrrr} %{l|ccccccc} 
   & \Delta_{6i+3} & \Delta_{6i+4} & \Delta_{6i+5} & \Delta_{6i+6} & \Delta_{6i+7} & \Delta_{6i+8} \\
 \hline 
F_{6i+1,1,2,n-4-6i} & 1 & -1 & 0 & 0 & 0 & 0   \\
F_{6i+2,1,1,n-4-6i} & 2 & -1 & 0 & 0 & 0 & 0  \\
F_{6i+1,2,2,n-5-6i} & 2 & 0 & -1 & 0 & 0 & 0   \\
F_{6i+4,2,2,n-8-6i} & 0 & -1 & 0 & 2 & 0 & -1  \\
F_{6i+5,1,1,n-7-6i} & 0 & 0 & -1 & 2 & -1 & 0  \\
F_{6i+5,1,2,n-8-6i} & 0 & 0 & -1 & 1 & 1 & -1  
\end{array} 
\end{align*}
%The determinant is 1=1*(-1)*(-1)
When $i=0$, the matrix is slightly modified:
\begin{align*}
\begin{array}{l| rrrrrrr} 
   & \Delta_{3} & \Delta_{4} & \Delta_{5} & \Delta_{6} & \Delta_{7} & \Delta_{8} \\
 \hline 
F_{1,1,2,n-4} & 2 & -1 & 0 & 0 & 0 & 0   \\
F_{1,2,2,n-5} & 2 & 1 & -1 & 0 & 0 & 0   \\
F_{4,1,2,n-7} & 1 & -1 & 1 & 1 & -1 & 0  \\
F_{5,1,1,n-7} & 0 & 0 & -1 & 2 & -1 & 0  \\
F_{4,2,2,n-8} & 0 & -1 & 0 & 2 & 0 & -1  \\
F_{5,1,2,n-8} & 1 & 0 & -1 & 1 & 1 & -1  
\end{array} 
\end{align*}
%The determinant is 2!
Finally, the intersection pairing of $\N_{t-1}$ with $\Delta_{6t-3}$, $\Delta_{6t-2}$, $\Delta_{6t-1}$, $\Delta_{6t}$
is
\begin{align*}
\begin{array}{l | rrrr} 
   & \Delta_{6t-3} & \Delta_{6t-2} & \Delta_{6t-1} & \Delta_{6t}  \\
 \hline 
F_{6t-5,1,2, 6t+2} & 1 & -1 & 0 & 0 \\
F_{6t-4,1,1, 6t+2} & 2 & -1 & 0 & 0   \\
F_{6t-5,2,2,6t+1}  & 2 & 0 & -1 & 0   \\
F_{6t-1,1,1,6t-1}   & 0 & 0 & -2 & 2    \\
\end{array} 
\end{align*}
%The determinant is (-4)

\noindent
{\em Remaining cases:} %The remaining cases are proved as follows. 
If $n=12t+3$, we take $$\N_{t-1}=\{F_{6t-5,1,2, 6t+5}, F_{6t-4,1,1, 6t+5},  F_{6t-1,1,1, 6t+2}, F_{6t-4, 1, 2, 6t+4}, F_{6t-1, 1, 2, 6t+1}\}.$$
Then the intersection pairing of $\N_{t-1}$ with $\Delta_{6t-3}$, $\Delta_{6t-2}$, $\Delta_{6t-1}$, $\Delta_{6t}$, $\Delta_{6t+1}$
is nondegenerate.
\begin{comment}
\begin{align*}
\begin{array}{l | rrrrr} 
   & \Delta_{6t-3} & \Delta_{6t-2} & \Delta_{6t-1} & \Delta_{6t} & \Delta_{6t+1} \\
 \hline 
F_{6t-5,1,2, 6t+5} & 1 & -1 & 0 & 0 & 0  \\
F_{6t-4,1,1, 6t+5} & 2 & -1 & 0 & 0  & 0 \\
F_{6t-1,1,1,6t+2}  & 0 & 0 & -1 & 2 & -1 \\
F_{6t-4, 1, 2, 6t+4} & 1 & 1 & -1 & 0 & 0 \\
F_{6t-1, 1, 2, 6t+1} & 0 & 0 & -1 & 1 & 0 
\end{array} 
\end{align*}
\end{comment}
If $n=12t+6$, we take
\begin{multline*}
\N_{t-1}=\{F_{6t-5,1,2, 6t+8}, F_{6t-4,1,1, 6t+8},  F_{6t-5,2,2, 6t+7}, F_{6t-1, 1, 1, 6t+5}, \\
F_{6t-2, 1, 2, 6t+5}, F_{6t-2,2,2, 6t+4}, F_{6t+1,2,2,6t+1}\}.
\end{multline*}
Then the intersection pairing of $\N_{t-1}$ with $\Delta_{6t-3}$, $\Delta_{6t-2}$, $\Delta_{6t-1}$, $\Delta_{6t}$, 
$\Delta_{6t+1}$, $\Delta_{6t+2}$, $\Delta_{6t+3}$
is nondegenerate.
\begin{comment}
\begin{align*}
\begin{array}{l | rrrrrrr} 
   & \Delta_{6t-3} & \Delta_{6t-2} & \Delta_{6t-1} & \Delta_{6t} & \Delta_{6t+1} & \Delta_{6t+2} & \Delta_{6t+3} \\
 \hline 
F_{6t-5,1,2, 6t+8} & 1 & -1 & 0 & 0 & 0 & 0 & 0 \\
F_{6t-4,1,1, 6t+8} & 2 & -1 & 0 & 0 & 0 & 0 & 0 \\
F_{6t-5,2,2, 6t+7}  & 2 & 0 & -1 & 0 & 0 & 0 & 0 \\
F_{6t-1, 1, 1, 6t+5} & 0 & 0 & -1 & 2 & -1 & 0 & 0 \\
F_{6t-2, 1, 2, 6t+5} & 0 & -1 & 1 & 1 & -1 & 0 & 0 \\
F_{6t-2,2,2, 6t+4} & 0 & -1 & 0 & 2 & 0 & -1 & 0 \\ 
F_{6t+1,2,2,6t+1} & 0 & 0 & 0 & 0 & -2 & 0 & 2 
\end{array} 
\end{align*}
\end{comment}
Finally, if $n=12t+9$, we take
\begin{multline*}
\N_{t-1}=\{F_{6t-5,1,2, 6t+11}, F_{6t-4,1,1, 6t+11},  F_{6t-5,2,2, 6t+10}, F_{6t-1, 1, 1, 6t+8}, \\
F_{6t-2, 1, 2, 6t+8}, F_{6t-2,2,2, 6t+7}, F_{6t+1,1,2,6t+5}, F_{6t+1,2,2,6t+4} \}.
\end{multline*}
Then the intersection pairing of $\N_{t-1}$ with $\Delta_{6t-3}$, $\Delta_{6t-2}$, $\Delta_{6t-1}$, $\Delta_{6t}$, 
$\Delta_{6t+1}$, $\Delta_{6t+2}$, $\Delta_{6t+3}$, $\Delta_{6t+4}$
is nondegenerate.
\begin{comment}
\begin{align*}
\begin{array}{l | rrrrrrrr} 
   & \Delta_{6t-3} & \Delta_{6t-2} & \Delta_{6t-1} & \Delta_{6t} & \Delta_{6t+1} & \Delta_{6t+2} & \Delta_{6t+3} & \Delta_{6t+4} \\
 \hline 
F_{6t-5,1,2, 6t+11} & 1 & -1 & 0 & 0 & 0 & 0 & 0 & 0 \\
F_{6t-4,1,1, 6t+11} & 2 & -1 & 0 & 0 & 0 & 0 & 0 & 0\\
F_{6t-5,2,2, 6t+10}  & 2 & 0 & -1 & 0 & 0 & 0 & 0  & 0\\
F_{6t-1, 1, 1, 6t+8} & 0 & 0 & -1 & 2 & -1 & 0 & 0 & 0 \\
F_{6t-2, 1, 2, 6t+8} & 0 & -1 & 1 & 1 & -1 & 0 & 0 & 0\\
F_{6t-2,2,2, 6t+7} & 0 & -1 & 0 & 2 & 0 & -1 & 0 & 0\\ 
F_{6t+1,1,2,6t+5} & 0 & 0 & 0 & 0 & -1 & 1 & 1 & -1 \\ 
F_{6t+1,2,2,6t+4} & 0 & 0 & 0 & 0 & -1 & 0 & 2 & -1 \\ 
\end{array} 
\end{align*}
\end{comment}
\renewcommand{\qedsymbol}{}
\end{proof}

\end{proof}

\begin{remark} 
We remark that 
$2\psi-2\Delta-\sum_{3\mid k}\Delta_{k}$ is clearly F-nef: 
It intersects every F-curve non-negatively and has degree $0$ precisely on F-curves congruent to $(2,2,1,1)$ modulo $3$.

The divisor $D:=2\psi-2\Delta-\sum_{3\mid k}\Delta_{k}$ of Theorem \ref{T:3} can be 
rewritten as 
$$2\bigl(K_{\M_{0,n}}+\sum_{3\nmid k}\Delta_{k}+\frac{1}{2}\sum_{3\mid k}\Delta_{k}\bigr).$$
%It is a multiple of a log canonical divisor on $\tM{0,n}$. 
Because $D/2$ is F-nef and of the form 
$K_{\M_{0,n}}+G$ where $\Delta-G\geq 0$, a theorem of Farkas and Gibney \cite[Theorem 4]{FG}
implies that $D$ is nef. Nef and big log canonical divisors are expected to be semiample. Whether this is 
the case is still an open question.

In a similar vein, if we allow ourselves to use the results of higher-dimensional birational geometry, such as the Contraction Theorem, as in \cite{keel-mckernan, FG}, we
obtain the following.
\begin{prop}\label{P:prop-FG}
Suppose $p\mid n$. Then the divisor 
\[
\psi-\Delta-\frac{1}{2}\sum_{p\mid k} \Delta_k=K_{\M_{0,n}}+\sum_{p\nmid k}\Delta_{k}
+\frac{1}{2}\sum_{p\mid k}\Delta_{k}
\]
is nef. 
\end{prop}
\begin{proof}
Indeed, the divisor in question is easily seen to be F-nef. A theorem of Farkas and Gibney \cite[Theorem 4]{FG} 
finishes the proof. 
\end{proof}

\noindent
We note that the divisor of Proposition \ref{P:prop-FG} clearly lies on the boundary of 
$\Nef(\tM{0,n})$. However, it does not generate an extremal ray for $p\geq 4$. We also note
that if $p\geq 3$ is prime, then 
$$\psi-\Delta-\frac{1}{2}\sum_{p\mid k} \Delta_k\sim f^*_{n,p}\left(\frac{8p^2}{p^2-1}\lambda-\delta_{\irr}\right).$$
\end{remark}

\subsection{Extensions and examples}%{Extremal rays of $\Nef(\tM{0,n})$ from the cyclic covering construction}

We believe that Theorem \ref{T:3} admits a generalization that produces a manifold of new extremal
rays of $\Nef(\tM{0,n})$ coinciding with extremal rays of the F-cone. We prove one such generalization in the 
case $p=5$ below. Proving that an extremal ray
of the F-cone is actually generated by a nef divisor serves two purposes. On the one hand, it brings
us closer to the proof of the F-conjecture. On the other hand, it delineates the region of the F-cone where 
one should look (or not look) for counterexamples. 

Recall that the Hodge class $\lambda$
gives a measure of variation in moduli for one-parameter 
families of stable curves: As long as the (normalization of the) members
of the family vary nontrivially, the degree of $\lambda$ is positive. A related observation holds for the 
divisor class $\delta_{\irr}$: if the degree of $\delta_{\irr}$ is positive on a one-parameter family of stable curves, then
the variation in moduli in the family is nontrivial and so the degree of $\lambda$ is also positive.
(This property clearly fails for all other boundary divisor classes, as the divisor $\Delta_{a-1}$ and the curve $T_a$ illustrate; see Construction \ref{S:test-curves}).

For every closed subvariety $Z\subset \M_g$, there exists a positive constant $c$ such
that $c\lambda-\delta_{\irr}$ lies on the boundary of $\Nef(Z)$. For $Z=\M_g$, it is well-known that $12\lambda-\delta_{\irr}$ is nef. 
In fact, it generates an extremal ray of $\Nef(\M_g)$; see \cite{faber, GKM}. 
For $Z=\overline{\{\text{cyclic trigonal curves}\}}$,  
Theorem \ref{T:3} shows that $9\lambda-\delta_{\irr}$ is an extremal ray of $\Nef(Z)$.

In a similar vein, there is evidence
that for every prime $p\geq 3$ and $j\in \{1,\dots, p-1\}$, 
the divisor $2p^2\lambda_{n,p}(j)-\delta_\irr$ generates an extremal
ray of $\Nef(\tM{0,n})$. 
 %In a similar vein, a suitable linear combination of $\lambda_p(j)$ and $\delta_{\irr}$ lie on the boundary 
%of the nef cone of the locus of $\mu_p$-covers in $\M_{(n-2)(p-1)/2}$.  
For $p=5$, this observation is formalized in the following proposition.
%\footnote{We take $p$ to be prime for simplicity. The case 
%of composite $p$ should be completely analogous.} 
\begin{prop}\label{T:p=5} Suppose $5\mid n$. 
The divisor classes 
\begin{align*}
f^*_{n,5}(50\lambda_{n,5}(1)-\delta_{\irr})
&=4\psi-4\sum_{k\equiv 1,4\hspace{-0.5pc}\mod 5}\Delta_k-6\sum_{k\equiv 2,3\hspace{-0.5pc} \mod 5}\Delta_k-5\sum_{5\mid k}\Delta_k, \\
f^*_{n,5}(50\lambda_{n,5}(2)-\delta_{\irr})
&=6\psi-6\sum_{k\equiv 1,4\hspace{-0.5pc}\mod 5}\Delta_k-4\sum_{k\equiv 2,3\hspace{-0.5pc}\mod 5}\Delta_k-5\sum_{5\mid k}\Delta_k
\end{align*}
are nef on $\M_{0,n}$. 
\end{prop}

\begin{remark}
It is almost certain that both of these divisors in fact generate an extremal ray of $\Nef(\tM{0,n})$ (for $5\mid n$).
This is verified for $n=10$ below and can be verified by hand for $n=15$. 
We omit a verification in the general case, which, in the absence of 
any further insight, would be a tedious exercise in linear algebra. 
\end{remark}

\begin{remark} It is not hard to see that the divisor 
$$f^*_{n,5}(50\lambda_{n,5}(2)-\delta_{\irr})
=6\psi-6\sum_{k\equiv 1,4\hspace{-0.5pc}\mod 5}\Delta_k-4\sum_{k\equiv 2,3\hspace{-0.5pc}\mod 5}\Delta_k-5\sum_{5\mid k}\Delta_k$$
is not log canonical already for $n\geq 25$.
\begin{comment}
The slope of coefficients of $\Delta_k$ and $\Delta_2$ is at least $\left(k(n-k)-2(n-1)\right)/(2(n-2)-(n-1))$ in
$c(K+\sum c_i \Delta_i)$. On the other hand in the divisor under consideration, we have that the slope is 
$(6k(n-k)-4(n-1))/(8n-20)$. This translates into $(6n-14)\geq k(n-k)$. When $n=20$, take $k=8$, it is still OK. 
But when $n=20$ and $k=12$, we are done. 
\end{comment}
 Therefore, the techniques of \cite{keel-mckernan, FG} cannot be used to establish its nefness for $n\geq 25$.
This divisor also does not appear to be a conformal blocks divisor. 
\end{remark}

\begin{proof}
First, we verify that $f^*_{n,5}(50\lambda_{n,5}(j)-\delta_{\irr})$ has non-negative degree on every F-curve
by considering all possible F-curves modulo $5$: An F-curve of type with $(a,b,c,d)$ with $5\mid abcd$ is
easily seen to intersect the divisors in question positively. An F-curve of type $(1,4,1,4)$ modulo $5$ intersects
$f^*_{n,5}(50\lambda_{n,5}(1)-\delta_{\irr})$ in degree $0$, and $f^*_{n,p}(50\lambda_{n,5}(2)-\delta_{\irr})$ in 
degree $10$. An F-curve of type $(2,3,2,3)$ modulo $5$ intersects
$f^*_{n,5}(50\lambda_{n,5}(1)-\delta_{\irr})$ in degree $10$, and $f^*_{n,p}(50\lambda_{n,5}(2)-\delta_{\irr})$ in 
degree $0$. All remaining F-curves intersect divisors in question positively.

From now on, the proof parallels that of Theorem \ref{T:3} above.
We begin by observing that
classes $50\lambda_{n,5}(j)-\delta_{\irr}$ are superadditive under the operation of normalization 
along generic nodes. Therefore, it suffices to treat the 
case of a generically smooth family.   
We consider a family $(\X\ra B; \{\sigma_i\}_{i=1}^m)$ with sections 
$\{\sigma_i\}_{i=1}^m$ endowed with weights $d_i\in \{1,2,3,4\}$. (This indicates that a weighted cyclic covering morphism is lurking in the background.)
For $k\in \{1,2,3,4\}$, 
let $D_k:=\sum_{I,J}\Delta_{I,J}$, where the sum is taken over partitions $I\cup J=\{1,\dots,m\}$ such that 
$\sum_{i\in I} d_i=k \mod 5$. We also set $\Psi_k=\sum_{i:\, d_i=k} \psi_i$. 
% where the sum is taken over section of weight $k$.
Then in view of Proposition \ref{P:new-formula}, the two divisors $f^*_{n,5}(50\lambda_{n,5}(j)-\delta_{\irr})$, for $j=1$ and $j=2$, become incarnations of
the same divisor on the moduli stack of pointed curves with weights. Namely, we have the divisor
\[
Q:=4(\Psi_1+\Psi_4)+6(\Psi_2+\Psi_3)-4(D_1+D_4)-6(D_2+D_3)-5D_5.
\]
For a generically smooth stable $m$-pointed family of rational curves with $m\geq 5$, 
we always have the inequality  
$4\psi-6\Delta\geq 0$, which follows immediately from the identity $(m-1)\psi=\sum_k k(m-k)\Delta_k$ on $\M_{0,m}$. 
It follows that $Q$ is non-negative on such families. 
The generically smooth stable $4$-pointed families of rational curves are precisely 
the F-curves, and $Q$ is non-negative on such families by the inspection above. The proposition follows. 

%The next key ingredient is supplied
%by Proposition \ref{P:new-formula}. Using it, we rewrite $50^2\lambda_{\vec{d},p}(j)-\delta_{\irr}$  
%\begin{align*}
%2p^2\lambda_{\vec{d},p}(j)= %2p^2\lambda_{\vec{d},p}(j)-\delta_{\irr}=
%\sum_{i=1}^n c(d_i)(p-c(d_i))\psi_i-\sum_{I,J}c(d(I))(p-c(d(I)))\Delta_{I,J},
%\end{align*}
%where we use the notation $c(n)=\modp{jn}{p}$. By Proposition \ref{P:pullback}, 
%$$\delta_{\irr}=p\Delta_{I,J}.$$
\end{proof}

We proceed to show that the divisors obtained from the 
cyclic covering morphisms span all extremal rays of 
$\Nef(\tM{0,n})$ for $n=9$ and $n=10$.
All computations with convex polytopes were done using {\tt lrs} \cite{lrs}. 
\subsubsection{Nef cone of $\tM{0,9}$} Let $\vec{d}=(1,1,1,1,1,1,1,1,2)$.
The extremal rays of $\Nef(\tM{0,9})$ are as follows:
\begin{align*}
D_1 &=\Delta_2+3\Delta_3+6\Delta_4 \sim (f^{S_9}_{\vec{d},2})^* (10\lambda-\delta_{\irr}-2\delta_{\red}),  \\
D_2 &=3\Delta_2+3\Delta_3+4\Delta_4 \sim (f^{S_9}_{\vec{d},2})^* (\lambda), \\
D_3 &=\Delta_2+3\Delta_3+2\Delta_4 \sim (f_{9,3})^* (\lambda), \\
D_4 &=\Delta_2+\Delta_3+2\Delta_4 \sim (f_{9,3})^* (9\lambda-\delta_{\irr}).
\end{align*}
We note that the divisor class $10\lambda-\delta_{\irr}-2\delta_{\red}$ generates an extremal 
ray of $\Nef(\M_g)$ for every $g\geq 2$ by \cite{GKM}.
\subsubsection{Nef cone of $\tM{0,10}$}
The F-curves on $\tM{0,10}$ and their coordinates in the standard basis $\Delta_2, \Delta_3, \Delta_4, \Delta_5$ 
are as follows:
\begin{table}[h] 
\renewcommand{\arraystretch}{1.1}
\begin{tabular}{lll}
$C_{1} =F_{7,1,1,1}=(3, -1, 0, 0)$  & 
$C_{2} =F_{6,2,1,1} =(0,2,-1,0)$ &
$C_{3} =F_{5,3,1,1} =(1,-1,2,-1)$ \\
$C_{4} =F_{5,2,2,1} =(-2,2,1,-1)$ &
$C_{5} =F_{4,4,1,1} =(1,0,-2,2)$ &
$C_{6} =F_{4,3,2,1} =(-1,0,0,1)$ \\
$C_{7} =F_{4,2,2,2} =(-3,0,2,0)$ &
$C_{8} =F_{3,3,3,1} =(0,-3,3,0)$ &
$C_{9} =F_{3,3,2,2} =(-2,-2,1,2)$
\end{tabular}
\end{table}

Using {\tt lrs}, we compute the extremal rays of the F-cone. 
Using cyclic covering morphisms, we prove that every extremal ray is generated by nef divisors.
The results are listed in Table \ref{table-10}.
We refer to \cite{agss, ags} for background
on $\sl_n$ and $\sl_2$ conformal blocks divisors on $\M_{0,n}$.

\begin{table}[h] 
\label{table-10}
\renewcommand{\arraystretch}{1.6}
\begin{tabular}{|l | l | l | l | }
\hline
\multirow{2}{2in}{{\small Extremal ray of $\Nef(\tM{0,10})$}} & \multirow{2}{1in}{{\small Orthogonal F-curves}} & \multirow{2}{1in}{
{\small Cyclic covering interpretation}} & 
\multirow{2}{0.9in}{{\small Conformal blocks interpretation}}   \\
& & &  \\
\hline
$4\Delta_{2}+6\Delta_{3}+6\Delta_{4}+7\Delta_{10}$ & $C_{7}, C_{8}, C_{9}$ &
 $50\lambda_{10,5}(2)-\delta_{\irr}$ (Prop. \ref{T:p=5}) & N/A \\
 \hline
 $2\Delta_{2}+6\Delta_{3}+6\Delta_{4}+5\Delta_{5}$ & $C_{1}, C_{5}, C_{8}$ &
 $(f_{(0,1,\dots,1),3}^{S_9})^*\lambda\sim \lambda_{10,10}(3)$ &  $\DD(\sl_{10},1,w_{3}^{10})$  \\
 \hline
 $4\Delta_{2}+3\Delta_{3}+6\Delta_{4}+4\Delta_{5}$ &  $C_{2},C_{4}, C_{5},C_{6},C_{7}, C_{9}$ &
 $f_{10,2}^{*}\lambda=\lambda_{10,2}(1)\sim \lambda_{10,10}(5)$ &
 \multirow{2}{0.9in}{$\DD(\sl_{10},1, w_{5}^{10})$ $\sim \DD(\sl_{2},k,k^{10})$} \\
 & & & \\
 \hline
 $2\Delta_{2}+6\Delta_{3}+12\Delta_{4}+11\Delta_{5}$ & $C_{1}, C_{2}, C_{5}$ & 
\multirow{2}{1.5in}{$f^{*}_{10,2}(10\lambda-\delta_{\irr}-2\delta_{\red})$ $\sim 50\lambda_{10,5}(1)-\delta_{\irr}$} & $\DD(\sl_{2},3,1^{10})$  \\ 
& & & \\
\hline
 $2\Delta_{2}+3\Delta_{3}+3\Delta_{4}+5\Delta_{5}$ & $C_{3}, C_{4}, C_{7}, C_{8}$ &
   $\lambda_{10,5}(2)\sim \lambda_{10,10}(4)$ & $\DD(\sl_{10},1,w_{4}^{10})$\\
   \hline
    $\Delta_{2}+3\Delta_{3}+3\Delta_{4}+4\Delta_{5}$ & $C_{1}, C_{3}, C_{8}$ &
    $f^{*}_{10,2}(12\lambda-\delta_{0})$ & $\DD(\sl_{2},2,1^{10})$ \\
    \hline
  $\Delta_{2}+3\Delta_{3}+6\Delta_{4}+10\Delta_{5}$ & $C_{1}, C_{2}, C_{3}, C_{4}$ &
 $\lambda_{10,5}(1)\sim \lambda_{10,10}(2)$ &  $\DD(\sl_{10},1,w_{2}^{10})$ \\
 \hline
\end{tabular}
\smallskip
\caption{Extremal rays of $\Nef(\tM{0,10})$ via the cyclic covering morphisms.}
\label{T:table-dangling}
\end{table}

\begin{comment}
Consider all F-curves modulo $5$. On $\tM{5m}$, we compute 
\begin{align*}
\lambda_5(1)\cdot F_{3,3,2,2} &=2, \ \delta_{\irr}\cdot F_{3,3,2,2} &=5\times 2=10, \\
\lambda_5(1)\cdot F_{4,1,4,1} &=1 \ \delta_{\irr}\cdot F_{4,1,4,1} &=5\times 2=10, \\
\lambda_5(1)\cdot F_{2,1,1,1} &=0, \ \delta_{\irr}\cdot F_{2,1,1,1} &=5\times 2=0, \\
\lambda_5(1)\cdot F_{3,4,4,4} &=1, \ \delta_{\irr}\cdot F_{3,4,4,4} &=5\times 2=0, \\
\lambda_5(1)\cdot F_{4,2,2,2} &=1, \ \delta_{\irr}\cdot F_{4,2,2,2} &=5\times 2=0, \\
\lambda_5(1)\cdot F_{1,3,3,3} &=1, \ \delta_{\irr}\cdot F_{1,3,3,3} &=5\times 2=0.
\end{align*}
It follows that $10\lambda_5(1)-\delta_{\irr}$ is F-nef.
\begin{align*}
\lambda_5(2)\cdot F_{3,3,2,2} &=1, \ \delta_{\irr}\cdot F_{3,3,2,2} &=5\times 2=10, \\
\lambda_5(2)\cdot F_{4,1,4,1} &=2 \ \delta_{\irr}\cdot F_{4,1,4,1} &=5\times 2=10, \\
\lambda_5(2)\cdot F_{2,1,1,1} &=1, \ \delta_{\irr}\cdot F_{2,1,1,1} &=5\times 2=0, \\
\lambda_5(2)\cdot F_{3,4,4,4} &=1, \ \delta_{\irr}\cdot F_{3,4,4,4} &=5\times 2=0, \\
\lambda_5(2)\cdot F_{4,2,2,2} &=0, \ \delta_{\irr}\cdot F_{4,2,2,2} &=5\times 2=0, \\
\lambda_5(2)\cdot F_{1,3,3,3} &=0, \ \delta_{\irr}\cdot F_{1,3,3,3} &=5\times 2=0.
\end{align*}
It follows that $10\lambda_5(1)-\delta_{\irr}$ is F-nef. Note that 
\end{comment}

\section{Computing degrees of eigenbundles $\EE_j$}
\label{S:appendix}
In this section, we collect main technical results concerning the eigenbundles $\EE_j$ 
defined in Section \ref{S:eigen-hodge-classes}. These results are well-known. 
The eigenbundle decomposition of 
the Hodge bundle 
over a family of cyclic covers of $\PP^1$ with $4$ branch points has been considered in 
\cite[Section 3]{bouw-moller} and 
\cite{eskin-kontsevich-zorich}. %where the ranks and the degrees of the eigenbundles are computed.

We have decided to include these computations for two reasons. One reason is completeness: with this 
section the paper becomes essentially self-contained. The second reason is that we work exclusively in the algebraic
category, and so all of the results continue to hold in sufficiently high positive characteristic.

\subsection{Weight $j$ forms on $\mu_p$-covers of $\PP^1$ with $3$ branch points}

\begin{definition}[Branched covers of a $3$-pointed $\PP^1$]\label{D:3-branched}
We define $C(a,b)$ to be the normalization of the curve defined by the equation
$
y^p=x^a(x-1)^b.
$
The resulting branched cover\footnote{
No divisibility conditions on $a$ and $b$ are imposed; in particular, $C(a,b)$ can be disconnected.} 
$\pi\co C(a,b) \ra \PP^1$ has branch points at $0$, $1$, and $\infty$. 
Set $c=\modp{p-a-b}{p}$. 
We consider the reduced divisors $[0]:=\pi^{-1}(0)$, $[1]:=\pi^{-1}(1)$, and $[\infty]:=\pi^{-1}(\infty)$.
Evidently, $\deg [0]=\gcd(a,p)$, $\deg [1]=\gcd(b,p)$, and $\deg [\infty]=\gcd(c,p)$. Note
that by symmetry $C(a,b)=C(a,c)=C(b,c)$.
\end{definition}

\begin{lemma}\label{L:weight-j-3-branched}
Let $C=C(a,b)$ be as in Definition \ref{D:3-branched}.  
The weight spaces of $\HH^0(C,\omega_C)$ with respect to the $\mu_p$-action are computed as follows. Consider the unique 
integers $k$ and $\ell$ satisfying
\begin{align*}
aj-\gcd(a,p) &=kp+\modp{aj-\gcd(a,p)}{p}, \\
bj-\gcd(b,p) &=\ell p+\modp{bj-\gcd(b,p)}{p},
\end{align*}
and define 
$\omega:=y^jdx/x^{k+1}(x-1)^{\ell+1}.$
\begin{enumerate}
\item[(a)] If $\modp{aj}{p}=0$ or $\modp{bj}{p}=0$, then $\HH^0(C,\omega_C)_j=(0)$.
\item[(b)] If $\modp{aj}{p}+\modp{bj}{p}\geq p$, then $\HH^0(C,\omega_C)_j=(0)$.
\item[(c)] If $\modp{aj}{p}\modp{bj}{p}>0$ and $\modp{aj}{p}+\modp{bj}{p}\leq p-1$, \\ then $\HH^0(C,\omega_C)_j=\spn\{\omega\}$.
%\item[(b)] If $\modp{aj-\gcd(a,p)}{p}+\modp{bj-\gcd(b,p)}{p}\leq p-1$, then $\HH^0(C,\omega_C)_j$ is one-dimensional 
%\item[(a)] If $\modp{aj}{p}+\modp{bj}{p}+\modp{cj}{p}=2p$, then $\HH^0(C,\omega_C)_j=(0)$.
%\item[(a)] If $\modp{aj}{p}=0$ or $\modp{bj}{p}=0$, then $\HH^0(C,\omega_C)_j=(0)$.
%\item[(c)] If $\modp{aj}{p}+\modp{bj}{p}+\modp{cj}{p}=p$, then $\HH^0(C,\omega_C)_j=\CC\{\omega\}$. 
%is one-dimensional and is spanned by $\omega$.
\end{enumerate}
\end{lemma}
\begin{proof} 
Every rational weight $j$ form on $C$ looks like $y^jdx/g(x)$.
We begin by writing down the relevant divisors on $C(a,b)$: 
\begin{align*}%\tag{\dagger} %\label{equations-1}
(y) &=\frac{a}{\gcd(a,p)}[0]+\frac{b}{\gcd(b,p)}[1]-\left(\frac{a}{\gcd(a,p)}+\frac{b}{\gcd(b,p)}\right)[\infty], \\
(dx) &=\frac{p-\gcd(a,p)}{\gcd(a,p)} [0]+\frac{p-\gcd(b,p)}{\gcd(b,p)}[1]-\left(\frac{p}{\gcd(c,p)}+1\right)[\infty], \\
(x) &=\frac{p}{\gcd(a,p)}[0]-\frac{p}{\gcd(a,p)}[\infty], \\
(x-1) &=\frac{p}{\gcd(b,p)}[1]-\frac{p}{\gcd(b,p)}[\infty].
\end{align*}
\noindent
The key observation is that $\omega=y^jdx/x^{k+1}(x-1)^{\ell+1}$ is a rational form of weight $j$
which is regular on $\PP^1\smallsetminus \infty$ and has 
the least possible orders of vanishing along $[0]$ and $[1]$. Namely,
$\gcd(a,p)\ord_{0}(\omega)=\modp{aj-\gcd(a,p)}{p}$ and $\gcd(b,p)\ord_{1}(\omega)=\modp{bj-\gcd(b,p)}{p}$. 
Note that $\deg \omega=p-\gcd(a,p)-\gcd(b,p)-\gcd(c,p)$. 
If $\modp{aj}{p}=0$ or $\modp{bj}{p}=0$, then we see immediately that 
$\ord_{\infty}(\omega)<0$. It follows that $\HH^0(C,\omega_C)_j=(0)$. \\
If $\modp{aj}{p}+\modp{bj}{p}\geq p$, then 
$\modp{aj-\gcd(a,p)}{p}+\modp{bj-\gcd(b,p)}{p}\geq p-\gcd(a,p)-\gcd(b,p)$. Thus 
$\ord_{\infty}(\omega)<0$ and so $\HH^0(C,\omega_C)_j=(0)$.

Finally, if $\modp{aj}{p}\modp{bj}{p}>0$ and $\modp{aj}{p}+\modp{bj}{p}\leq p-1$, then we have
\begin{multline*}
\gcd(c,p)\ord_{\infty}(\omega)=p-\gcd(a,p)-\gcd(b,p)-\gcd(c,p) \\
-\modp{aj-\gcd(a,p)}{p}-\modp{bj-\gcd(b,p)}{p}\geq 1-\gcd(c,p).
\end{multline*}
Since $\gcd(c,p)\ord_{\infty}(\omega)$ is divisible by $\gcd(c,p)$, it follows that
$p-1\geq \gcd(c,p)\ord_{\infty}(\omega)\geq 0$. We conclude that $\omega$ 
is a unique (up to scaling) regular form of weight $j$.
\end{proof}

\begin{remark} In the situation of Lemma \ref{L:weight-j-3-branched} (c), we have 
$k=\lfloor aj/p\rfloor$ and $\ell=\lfloor bj/p\rfloor$.
\end{remark}

 \subsection{Weight $j$ forms on $\mu_p$-covers of $\PP^1$ with $4$ branch points}

\begin{definition}[Branched covers of a $4$-pointed $\PP^1$]\label{D:4-branched} Suppose $\lambda\neq 0,1,\infty$.
We define $C(a,b,c)$ to be the normalization of the curve defined by the equation
$$
y^p=x^a(x-1)^b(x-\lambda)^c.
$$
We consider the resulting branched cover\footnote{
No divisibility conditions on $a$, $b$, and $c$ are imposed; in particular, $C(a,b,c)$ can be disconnected.} 
 $\pi\co C(a,b,c) \ra \PP^1$ with branch points $0$, $1$, $\lambda$, $\infty$. 
Set $d=\modp{p-a-b-c}{p}$. 
We consider the reduced divisors $[0]:=\pi^{-1}(0)$, $[1]:=\pi^{-1}(1)$, $[\lambda]=\pi^{-1}(\lambda)$, and
$[\infty]:=\pi^{-1}(\infty)$.
Evidently, $\deg [0]=\gcd(a,p)$, $\deg [1]=\gcd(b,p)$, $\deg(\lambda)=\gcd(c,p)$, 
and $\deg [\infty]=\gcd(d,p)$. 
Note
that by symmetry $C(a,b,c)=C(a,c,d)=C(a,b,d)=C(b,c,d)$.
\end{definition}

\begin{lemma}\label{L:weight-j-4-branched}
Let $C=C(a,b,c)$ be as in Definition \ref{D:4-branched}.  
The weight spaces of $\HH^0(C,\omega_C)$ with respect to the $\mu_p$-action are as follows. 
Consider the unique 
integers $k$, $\ell$, and $m$ satisfying
\begin{align*}
aj-\gcd(a,p) &=kp+\modp{aj-\gcd(a,p)}{p}, \\
bj-\gcd(b,p) &=\ell p+\modp{bj-\gcd(b,p)}{p}, \\
cj-\gcd(c,p) &=m p+\modp{cj-\gcd(c,p)}{p},
\end{align*}
and define 
$
\omega:=y^jdx/x^{k+1}(x-1)^{\ell+1}(x-\lambda)^{m+1}.
$
If $\modp{aj}{p}\modp{bj}{p}\modp{cj}{p}=0$, then $\HH^0(C,\omega_C)_j=(0)$. In the  
case $\modp{aj}{p}\modp{bj}{p}\modp{cj}{p}>0$, we have
\begin{enumerate}
\item[(a)] If $\modp{aj}{p}+\modp{bj}{p}+\modp{cj}{p}\geq 2p$, then $\HH^0(C,\omega_C)_j=(0)$.
\item[(b)] If $p\leq \modp{aj}{p}+\modp{bj}{p}+\modp{cj}{p}\leq 2p-1$, then $\HH^0(C,\omega_C)_j=\spn\{\omega\}$.
\item[(c)] If $\modp{aj}{p}+\modp{bj}{p}+\modp{cj}{p}\leq p-1$, then $\HH^0(C,\omega_C)=\spn\{x\omega, (x-1)\omega\}$.
% is two-dimensional and is spanned by $y^jdx/x^{k-1}(x-1)^{\ell}(x-\lambda)^m$ 
%and $y^jdx/x^k(x-1)^{\ell-1}(x-\lambda)^m$.
\end{enumerate}
\end{lemma}
\begin{proof} 
%$(x)=\left(\frac{p}{\gcd(a,p)}-1\right)[0]+\left(\frac{p}{\gcd(b,p)}-1\right)[0]
%+\left(\frac{p}{\gcd(c,p)}-1\right)[0]
As in the proof of Lemma \ref{L:weight-j-4-branched}, we begin by computing 
\begin{align*}%\label{equations-1}
(y) &=\frac{a}{\gcd(a,p)}[0]+\frac{b}{\gcd(b,p)}[1]+\frac{c}{\gcd(c,p)}[\lambda]-\left(\frac{a}{\gcd(a,p)}+\frac{b}{\gcd(b,p)}+\frac{c}{\gcd(c,p)}\right)[\infty], \\
(dx) &=\frac{p-\gcd(a,p)}{\gcd(a,p)} [0]+\frac{p-\gcd(b,p)}{\gcd(b,p)}[1]+\frac{p-\gcd(c,p)}{\gcd(c,p)} [0]-\left(\frac{p}{\gcd(d,p)}+1\right)[\infty], \\
(x) &=\frac{p}{\gcd(a,p)}[0]-\frac{p}{\gcd(a,p)}[\infty], \\
(x-1) &=\frac{p}{\gcd(b,p)}[1]-\frac{p}{\gcd(b,p)}[\infty], \\
(x-\lambda)&=\frac{p}{\gcd(c,p)}[\lambda]-\frac{p}{\gcd(c,p)}[\infty].
\end{align*}

\noindent
Evidently, $\omega=y^jdx/x^{k+1}(x-1)^{\ell+1}(x-\lambda)^{m+1}$ is a form of weight $j$ which is 
regular on $\PP^1\smallsetminus\infty$
and has the least possible orders of vanishing along $[0]$, $[1]$, and $[\lambda]$. 
Namely, we have $\gcd(a,p)\ord_{0}(\omega)=\modp{aj-\gcd(a,p)}{p}, $
$\gcd(b,p)\ord_{1}(\omega)=\modp{bj-\gcd(b,p)}{p}, $ and 
$\gcd(c,p)\ord_{\lambda}(\omega)=\modp{cj-\gcd(c,p)}{p}$.
Note that $\deg \omega=2p-\gcd(a,p)-\gcd(b,p)-\gcd(c,p)-\gcd(d,p)$.

If $\modp{aj}{p}\modp{bj}{p}\modp{cj}{p}=0$, then 
$\ord_{\infty}(\omega)<0$. It follows that $\HH^0(C,\omega_C)_j=(0)$.
 
If $\modp{aj}{p}+\modp{bj}{p}+\modp{cj}{p}\geq 2p$, then 
$\modp{aj-\gcd(a,p)}{p}+\modp{bj-\gcd(b,p)}{p}+\modp{cj-\gcd(c,p)}{p}\geq 2p-\gcd(a,p)-\gcd(b,p)-\gcd(c,p)$. 
Thus $\ord_{\infty}(\omega)<0$ and so $\HH^0(C,\omega_C)_j=(0)$.

If $\modp{aj}{p}\modp{bj}{p}\modp{cj}{p}>0$ and $p\leq \modp{aj}{p}+\modp{bj}{p}+\modp{cj}{p}\leq 2p-1$, 
then $0\leq \gcd(d,p)\ord_{\infty}(\omega)\leq p-1$. 
It follows that $\omega$ is a unique (up to scaling) regular form of 
weight $j$.

Finally, if $\modp{aj}{p}\modp{bj}{p}\modp{cj}{p}>0$ and 
$0\leq \modp{aj}{p}+\modp{bj}{p}+\modp{cj}{p}\leq p-1$, then we have that 
$p\leq \gcd(d,p)\ord_{\infty}(\omega)\leq 2p-1$.
It follows that any other regular form of weight $j$ looks like $g(x)\omega$, where $g(x)$ is 
a rational function with at worst a single pole at $\infty$ and no other poles. The statement follows.
\end{proof}

\subsection{Universal $\mu_p$-cover over an F-curve}
We briefly recall the construction of the family of stable $\mu_p$-covers 
over $\M_{0,4}$ completing the family of smooth $\mu_p$-covers given by
 %$\PP^1\smallsetminus\{0,1,\infty\}$ defined by
\begin{align}\label{family-appendix}
C_{\lambda}: \quad y^{p}=x^{a}(x-1)^{b}(x-\lambda)^{c}, \quad \lambda\in \PP^1\smallsetminus\{0,1,\infty\}.
\end{align}
The construction parallels the global construction outlined in Section \ref{S:universal-p-cover}.
%of admissible covers of Harris-Mumford \cite{harris-mumford} 
%We proceeds in the following steps, the first of which constructs the universal family over $\M_{0,4}$.  
 \begin{construction}\label{construction}
Begin with a trivial family $\X:=\PP^1_x\times \PP^1_\lambda$ over $\PP^1_{\lambda}$ with $4$ sections \\
 $\Sigma_0: \{x=0\}$, $\Sigma_1:\{x=1\}$, $\Sigma_\infty: \{x=\infty\}$, and $\Sigma_\lambda:\{x=\lambda\}$. 
 Now perform the following steps:
  \begin{enumerate}
 \item[(1)] Blow up $3$ points where sections intersect; set $\X'=\mathrm{Bl\, } \X$.
 %obtain $4$ non-intersecting sections $\overline{\Sigma_0}$, $\overline{\Sigma_1}$, $\overline{\Sigma_\infty}$,
 %$\overline{\Sigma_\lambda}$ of self-intersection $(-1)$.
 \item[(2)] Make a base change $B\ra \PP^1_\lambda$ of degree $p$ totally ramified over $\lambda=0, 1, \infty$. \\
 Set $\Y:=\X'\times_{\PP^1_\lambda} B$ and $f\co \Y\ra \X$.
 \item[(3)] Take the degree $p$ branched cover of $\Y$ ramified %over the pullback of 
 %four sections by extracting the $p^{th}$ root of 
 over $f^{*}(a\Sigma_0+b\Sigma_1+c\Sigma_\lambda+d\Sigma_\infty)$.
 \item[(4)] Normalize the total space to obtain a family of stable curves $\Z \ra B$. 
 \end{enumerate}
Let $g\co \Z \ra \X$ be the resulting morphism.
The strict transforms on $\Z$ of the fibers of $\X$ over $\lambda=0,1,\infty$ are denoted $F_0$,
$F_1$, and $F_{\infty}$, respectively. The exceptional divisors of $g$ lying over 
$\lambda=0,1,\infty$ are denoted by $E_0$, $E_1$, and $E_{\infty}$. Note that $F_0=C(a+c,b)$ and
$E_0=C(a,c)$, etc.

 We call the family of stable curves obtained in Construction \ref{construction} 
 the {\em universal $\mu_p$-cover 
 of type $(a,b,c,d)$ over $\M_{0,4}$}. It is precisely the moving component of the pullback, by
 the cyclic covering morphism $f_{n,p}\co \M_{0,n}\ra \Mg{g}$, 
 of the universal family over $\Mg{g}$ to an F-curve of type $(a,b,c,d)$.
  \end{construction}

\begin{prop}\label{P:determinants}
Let $\EE$ be the Hodge bundle of the universal cyclic $\mu_p$-cover of type $(a,b,c,d)$ over $\M_{0,4}$. 
Then:
\begin{enumerate}
\item[(1)] The eigenbundle $\EE_j$ has rank $0$ for all $j\in\{0,1,\dots,p-1\}$ such that 
 \begin{align*}%\label{E:p}
 \modp{aj}{p}+\modp{bj}{p}+\modp{cj}{p}+\modp{dj}{p}=3p.
 \end{align*}
\item[(2)] The eigenbundle $\EE_j$ has rank $2$ and $\deg(\EE_j)=0$ for all $j\in\{0,1,\dots,p-1\}$  such that 
 \begin{align*}%\label{E:3p}
 \modp{aj}{p}+\modp{bj}{p}+\modp{cj}{p}+\modp{dj}{p}=p.
 \end{align*}
\item[(3)] The eigenbundle $\EE_j$ has rank $1$ for all $j\in\{0,1,\dots,p-1\}$ such that
 \begin{align*}%\label{E:2p}
 \modp{aj}{p}+\modp{bj}{p}+\modp{cj}{p}+\modp{dj}{p}=2p.
 \end{align*}
 In this case, we have
\begin{align*}
\deg(\EE_j)=\frac{1}{p}\min\{\modp{aj}{p},\modp{bj}{p},\modp{cj}{p},\modp{dj}{p}, 
\modp{-aj}{p},\modp{-bj}{p},\modp{-cj}{p},\modp{-dj}{p}\}.
\end{align*}
 \end{enumerate}
\end{prop}
\begin{proof} If $\modp{aj}{p}\modp{bj}{p}\modp{cj}{p}\modp{dj}{p}=0$, the statement follows immediately 
from Lemma \ref{L:weight-j-4-branched}. From now on, we assume 
that $\modp{aj}{p}\modp{bj}{p}\modp{cj}{p}\modp{dj}{p}>0$.

We consider the unique positive integers $k,\ell, m$ satisfying
\begin{align*}
aj-\gcd(a,p) &=kp+\modp{aj-\gcd(a,p)}{p}, \\
bj-\gcd(b,p) &=\ell p+\modp{bj-\gcd(a,p)}{p}, \\
cj-\gcd(c,p) &=mp+\modp{cj-\gcd(a,p)}{p}.\\
\end{align*} 
Set $f(x):=x^{k+1}(x-1)^{\ell+1}(x-\lambda)^{m+1}$ and $\omega:=y^jdx/f(x)$.

\begin{proof}[Proof of (1)] 
By Lemma \ref{L:weight-j-4-branched}, the fiber of $\EE_j$ at a point of $\PP^1\smallsetminus \{0,1,\infty\}$ is 
empty. The statement follows.
\renewcommand{\qedsymbol}{}
\end{proof}

\begin{proof}[Proof of (2)] 
By Lemma \ref{L:weight-j-4-branched}, the fiber of $\EE_j$ at a point of $\PP^1\smallsetminus \{0,1,\infty\}$ is 
spanned by $\omega_0:=x\omega$ and $\omega_1:=(x-1)\omega$. 
We extend $\omega_0\wedge \omega_1$ to a global rational section of $\det \EE_j$ 
and compute its zeros and poles. 

At $\lambda=0$, we have that $\omega_0=x\omega$ is
a regular form of weight $j$ on $F_0$ by Lemma \ref{L:weight-j-4-branched}.
%; we extend it to $E_0$ by zero. 
We now compute the extension of $\omega_0-\omega_1=\omega$ to $E_0$. 
%To do so, we compute in local coordinates near the generic point of $E_0$.
Set $x=\overline{x}\lambda$ and $\lambda=\eta^p$. Then $\overline{x}$ and $\eta$
are local 
coordinates near the generic point of $E_0$. 
The local equation of the branched cover in Step (3) of Construction \ref{construction} is
\begin{multline*}
y^p=x^a(x-1)^b(x-\lambda)^c %=\overline{x}^a\lambda^a(\overline{x}\lambda-1)^b(\overline{x}\lambda-\lambda)^c \\
=\lambda^{a+c}\overline{x}^a (\overline{x}\lambda-1)^b(\overline{x}-1)^c=\eta^{p(a+c)}\overline{x}^a (\overline{x}-1)^c(\overline{x}\lambda-1)^b.
\end{multline*}
It follows that after the normalization in Step (4) of Construction \ref{construction}, $E_0$ has equation 
$$z^p=\overline{x}^a(\overline{x}-1)^c(\overline{x}\lambda-1)^b,$$
where $z=y/\eta^{a+c}$. It follows that, modulo $d\eta$,
\begin{multline*}
\omega=y^jdx/f(x) %=\eta^{p} y^j \lambda d\overline{x}/\left(\lambda^{k+m+2}
%\overline{x}^{k+1}(\lambda \overline{x}-1)^{\ell+1}(\overline{x}-1)^{m+1}\right) \\
=\eta^{p+j(a+c)}z^{j}d\overline{x}/\bigl(\eta^{p(k+m+2)} \overline{x}^{k+1}(\lambda \overline{x}-1)^{\ell+1}(\overline{x}-1)^{m+1}\bigr)\\=\eta^{j(a+c)-p(k+m)-p}\omega',
\end{multline*}
where $\omega'$ restricts to a generator of $\HH^0(E_0,\omega_{E_0})_j$ 
by Lemma \ref{L:weight-j-3-branched}. %; we extend $\omega_0-\omega_1$ to $F_0$ by zero.
We conclude that $\omega_0\wedge \omega_1$ vanishes to order $j(a+c)-p(k+m+1)$ at $\lambda=0$.
Similarly, $\omega_0\wedge \omega_1$ vanishes to order 
$j(b+c)-p(\ell+m+1)$ at $\lambda=1$. 

Finally, we compute the vanishing order of $\omega_0\wedge \omega_1$ at $\lambda=\infty$. 
In terms of the local coordinate $\eta$ in the neighborhood of $\infty$ on $B$ we have $\lambda=1/\eta^p$. 
The equation of the branched cover in Step (3) of Construction \ref{construction} becomes
$
(y\eta^c)^p=x^a(x-1)^b(\eta^px-1)^c.
$
After the normalization in Step (4) of Construction \ref{construction}, 
the equation of $\Z$ in the neighborhood of $F_\infty$ is $z^p=x^a(x-1)^b$, where $z=y \eta^c$.
We compute that 
\begin{align*}
\omega_0-\omega_1=\omega %\lambda y^jdx/x^{(k+1)}(x-1)^{(\ell+1)}(x-\lambda)^{(m+1)}
=\eta^{p+pm-cj} z^jdx/x^{k+1}(x-1)^{\ell+1}(\eta^px-1)^{m+1}=\eta^{p(m+1)-cj}\omega',
\end{align*}
where $\omega'$ restricts to a generator of $\HH^0(F_\infty, \omega_{F_\infty})_j$
by Lemma \ref{L:weight-j-3-branched}.

In the neighborhood of $E_\infty$, we choose local coordinates $(\eta, u)$ such that
$\lambda=1/\eta^p$ and $x=u/\eta^p$.
The local equation of the branched cover in Step (3) of Construction \ref{construction} becomes
$
y^p %=(1/\overline{x}\eta^p)^a (1/\overline{x}\eta^p-1)^b (1/\overline{x}\eta^p-1/\eta^p)^c
%= \frac{1}{\eta^{p(a+b+c)}}\frac{(1-\overline{x}\eta^p)^b (1-\overline{x})^c}{\overline{x}^{(a+b+c)}} 
=\frac{1}{\eta^{p(a+b+c)}}(u-\eta^p)^b (u-1)^c u^a.
$
%where $u=1/\overline{x}$ (so that $u=x\eta^{p}$).

It follows that after the normalization in Step (4) of Construction \ref{construction},
the equation of $\Z$ in the neighborhood of $E_\infty$ is
$z^p= u^a(u-\eta^p)^b(u-1)^c ,$
where $z=y\eta^{a+b+c}$.
 It follows that, modulo $d\eta$,
\begin{align*}
\omega_0=x\omega %= x y^jdx/f(x)=u\eta^{-p} z^j\eta^{-j(a+b+c)} \eta^{-p}du/\left(u^{k+1}\eta^{-(k+1)p}(u/\eta^p-1)^{m+1}(u/\eta^p-1/\eta^p)^{\ell+1}\right) \\
=\eta^{p-j(a+b+c)+kp+mp+\ell p}z^jdu/\bigl(u^{k}(u-\eta^p)^{m+1}(u-1)^{\ell+1}\bigr).
\end{align*}
Summarizing, $\omega_0\wedge \omega_1$ vanishes to order 
$p(m+1)-cj+p-j(a+b+c)+kp+mp+\ell p=(2m+2)p+kp+\ell p-j(a+b+2c)$ at $\infty$. We conclude that 
$$\deg(\EE_j)=j(a+c)-p(k+m+1)+j(b+c)-p(\ell+m+1)+(2m+2)p+kp+\ell p-j(a+b+2c)=0.$$

\renewcommand{\qedsymbol}{}
\end{proof}

\begin{proof}[Proof of (3)] 
By Lemma \ref{L:weight-j-4-branched}, the fiber of $\EE_j$ at a point of $\PP^1\smallsetminus \{0,1,\infty\}$ is 
one-dimensional. 
Thus $\EE_j$ is a line bundle. We compute $c_1(\EE_j)$ by writing down a rational section of 
$\EE_j$ and counting its zeros and poles. 
We begin with the global rational $1$-form $\omega=y^jdx/f(x)$
that restricts to a generator of $\HH^0(C_\lambda, \omega_{\C_\lambda})_j$ for all $\lambda\in \PP^1\smallsetminus \{0,1,\infty\}$. 
Suppose, without loss of generality, that $0<\modp{aj}{p} \leq \modp{bj}{p} \leq \modp{dj}{p} \leq \modp{cj}{p}$.
We will treat only the case $\modp{cj}{p}+\modp{aj}{p}\geq p$. Then 
$p-\modp{cj}{p}=\min\{\modp{aj}{p},\modp{bj}{p},\modp{cj}{p},\modp{dj}{p}, 
\modp{-aj}{p},\modp{-bj}{p},\modp{-cj}{p},\modp{-dj}{p}\}$.

The key observation is that under our assumptions, the global $1$-form $\omega$ restricts
to the generator of $\HH^0(C_\lambda,\omega_{C_\lambda})_j$ for every $\lambda\neq \infty$. 
Indeed, the assumption $\modp{cj}{p}+\modp{aj}{p}\geq p$ implies that at $\lambda=0$, 
$\modp{(a+c)j}{p}+\modp{bj}{p}+\modp{dj}{p}=\modp{aj}{p}+\modp{cj}{p}+\modp{bj}{p}+\modp{dj}{p}-p=p$ and so $\omega$ restricts to a regular form 
on $F_0$ by Lemma \ref{L:weight-j-3-branched}. Similarly, $\omega$ restricts to a regular form on 
$F_1$. 

By Lemma \ref{L:weight-j-3-branched}, $\omega$ will restrict to a multiple of the generator
of $\HH^0(F_\infty, \omega_{F_\infty})_j$. 
It remains to compute the order of vanishing of $\omega$ at $\infty$. 
In terms of the local coordinate $\eta$ in the neighborhood of $\infty$ on $B$, we have $\lambda=1/\eta^p$. 
The equation of the branched cover in Step (3) of Construction \ref{construction} becomes 
$
(y\eta^c)^p=x^a(x-1)^b(\eta^px-1)^c.
$
After the normalization in Step (4) of Construction \ref{construction}, 
the equation of $\Z$ in the neighborhood of $F_\infty$ is $z^p=x^a(x-1)^b$, where $z=y \eta^c$.
We compute that 
\begin{align*}
\omega %\lambda y^jdx/x^{(k+1)}(x-1)^{(\ell+1)}(x-\lambda)^{(m+1)}
=\eta^{p+pm-cj} z^jdx/x^{k+1}(x-1)^{\ell+1}(\eta^px-1)^{m+1}=\eta^{p(m+1)-cj}\omega',
\end{align*}
where $\omega'$ restricts to a generator of $\HH^0(F_\infty, \omega_{F_\infty})_j$
by Lemma \ref{L:weight-j-3-branched}. It follows that 
$\omega$ vanishes to the order $pm+p-cj=p-\modp{cj}{p}$. 

Taking into the account the factor $1/p$ arising from the degree $p$ base change $B\ra \PP^1$
in Step (2) of Construction \ref{construction}, we conclude the proof.
\end{proof}
\renewcommand{\qedsymbol}{}
\end{proof}
\vspace{-2pc}
\begin{remark} Proposition \ref{P:determinants} is also proved in \cite[Theorem 1]{eskin-kontsevich-zorich}.
In particular, the computation similar to that of Proposition \ref{P:determinants} Part (3) appears in \cite[Section 2.4]{eskin-kontsevich-zorich}.
\end{remark}     

\subsection*{Acknowledgements} We would like to thank Aise Johan de Jong and Ian Morrison for helpful conversations. 

\bibliography{cyclic_references}
\bibliographystyle{amsalpha}
\end{document}